\documentclass[11pt]{article}
\usepackage[numbers,square]{natbib}
\usepackage{amsmath}
\usepackage{amsthm}
\usepackage{amsfonts}
\usepackage{amssymb}
\usepackage{siunitx}
\usepackage{commath}
\usepackage{graphicx}
\usepackage{xspace}
\usepackage{color}
\usepackage[lofdepth,lotdepth]{subfig}
\usepackage{stmaryrd}
\usepackage{url}
\usepackage{amsthm}
\usepackage{graphbox}
\usepackage{footnote}
\makesavenoteenv{tabular}
\makesavenoteenv{table}
\usepackage[hidelinks]{hyperref}
\hypersetup{colorlinks = true, urlcolor = blue,
  linkcolor = blue, citecolor = blue}
\usepackage{cleveref}
\usepackage{tikz}
\usetikzlibrary{trees,calc, positioning,decorations.markings,arrows,backgrounds,arrows.meta,
  bending,positioning,decorations.pathreplacing}

\usepackage[margin=0.7in]{geometry}
\makeatletter
\newcommand{\tnorm}{\@ifstar\@tnorms\@tnorm}
\newcommand{\@tnorms}[1]{%
  \left|\mkern-1.5mu\left|\mkern-1.5mu\left|
   #1
  \right|\mkern-1.5mu\right|\mkern-1.5mu\right|
}
\newcommand{\@tnorm}[2][]{%
  \mathopen{#1|\mkern-1.5mu#1|\mkern-1.5mu#1|}
  #2
  \mathclose{#1|\mkern-1.5mu#1|\mkern-1.5mu#1|}
}
\makeatother

\newcommand{\jump}[1]{\llbracket #1 \rrbracket}
\newcommand{\av}[1]{\{\!\!\{#1\}\!\!\}}

\newtheorem{lemma}{Lemma}
\newtheorem{corollary}{Corollary}
\newtheorem{remark}{Remark}

\title{Preconditioning a hybridizable discontinuous Galerkin method
  for Navier--Stokes at high Reynolds number}
\author{ A.~D.~Lindsay\thanks{Computational Frameworks, Idaho National
    Laboratory, Idaho Falls, ID, 83415, USA
    (\url{alexander.lindsay@inl.gov}),
    \url{https://orcid.org/0000-0002-6988-2123}}
  \and
  S.~Rhebergen\thanks{Department of Applied Mathematics, University of
    Waterloo, ON, Canada (\url{srheberg@uwaterloo.ca}),
    \url{https://orcid.org/0000-0001-6036-0356}}
  \and
  B.~S.~Southworth
  \thanks{Theoretical Division,
  Los Alamos National Laboratory,
  Los Alamos, NM 87545, USA (\url{southworth@lanl.gov}),
  \url{https://orcid.org/0000-0002-0283-4928}}
}
\begin{document}
\maketitle
\begin{abstract}
  We introduce a preconditioner for a hybridizable discontinuous
  Galerkin discretization of the linearized Navier--Stokes equations
  at high Reynolds number. The preconditioner is based on an augmented
  Lagrangian approach of the full discretization. Unlike standard
  grad-div type augmentation, however, we consider augmentation based
  on divergence-conformity. With this augmentation we introduce two
  different, well-conditioned, and easy to solve matrices to
  approximate the trace pressure Schur complement. To introduce a
  completely algebraic solver, we propose to use multifrontal sparse
  LU solvers using butterfly compression to solve the trace velocity
  block. Numerical examples demonstrate that the trace pressure Schur
  complement is highly robust in mesh spacing and Reynolds number and
  that the multifrontal inexact LU performs well for a wide range of
  Reynolds numbers.
\end{abstract}
\section{Introduction}
\label{s:introduction}

In this paper we consider the solution of the stationary
incompressible Navier--Stokes equations
\begin{subequations}
  \label{eq:ns}
  \begin{align}
    \label{eq:ns_a}
    -\nabla\cdot(2\mu\varepsilon(u)) + \nabla \cdot (u \otimes u) + \nabla p &= f & & \text{in } \Omega,
    \\
    \label{eq:ns_b}
    \nabla\cdot u &= 0 & & \text{in } \Omega,
    \\
    \label{eq:ns_c}
    u &= 0 & & \text{on } \partial \Omega,
    \\
    \label{eq:ns_d}
    \int_{\Omega} p \dif x &= 0, &&
  \end{align}
\end{subequations}
where $\Omega \subset \mathbb{R}^d$ is a polygonal (if $d=2$) or
polyhedral (if $d=3$) domain, $u : \Omega \to \mathbb{R}^d$ is the
velocity, $p : \Omega \to \mathbb{R}$ is the (kinematic) pressure,
$\varepsilon(u) := (\nabla u + (\nabla u)^T)/2$, $\mu > 0$ is the
constant kinematic viscosity, and $f : \Omega \to \mathbb{R}^d$ is the
external body force. We are interested in advection-dominated flows.

To discretize the Navier--Stokes equations, we consider a
\emph{hybridizable} discontinuous Galerkin method
\cite{Cockburn:2009}. Discontinuous Galerkin methods are known to be
stable, high-order accurate, and locally conservative methods for
advection-dominated flows \cite{Cockburn:2005}. Furthermore, if the
velocity is $H(\text{div};\Omega)$-conforming and exactly
divergence-free, then the DG discretization is also energy-stable
\cite{Cockburn:2007} and pressure-robust \cite{John:2017}. An
additional appealing property of hybridizable DG (HDG) methods is that
they allow for static condensation, a process in which local
degrees-of-freedom are eliminated from the system. Various HDG methods
have been introduced for incompressible flows, for example,
\cite{Cesmelioglu:2017,Cockburn:2014b,Egger:2013,Labeur:2012,Lederer:2018,Lehrenfeld:2016,Nguyen:2011,Qiu:2016}. We
will consider the HDG discretization presented in
\cite{Rhebergen:2018a,Rhebergen:2017}. This discretization results in
an exactly divergence-free and $H(\text{div};\Omega)$-conforming
velocity solution using standard DG spaces. The purpose of this paper
is to present a new preconditioner for this discretization.

To the best of our knowledge, the only preconditioners shown to be
robust in both mesh size and Reynolds number are augmented Lagrangian
(AL) type preconditioners (see, for example,
\cite{Benzi:2006,Benzi:2011b,Benzi:2011,Farrell:2019,Farrell:2021,Laakmann:2022}). It
is for this reason that we consider AL preconditioners in this
paper. However, some care must be taken when developing and applying
an AL preconditioner to the HDG discretization of
\cite{Rhebergen:2018a,Rhebergen:2017}. In most AL approaches for
Navier--Stokes a consistent penalty term
$-\gamma \nabla \nabla \cdot u$ is added to the left hand side of
\cref{eq:ns_a} (either at the algebraic or continuous level). However,
in the HDG method of \cite{Rhebergen:2018a,Rhebergen:2017} the
divergence constraint is a local constraint that is imposed exactly on
each cell. A consequence is that this grad-div penalization is local,
but more importantly, it vanishes after static condensation. As a
result, the Schur complement in the trace pressure remains unchanged
and, unfortunately, is still difficult to approximate. For this
reason, and further motivated by new linear algebra results on the
Schur complement of singularly perturbed saddle point problems we
present, we consider a penalty term based on
$H(\text{div};\Omega)$-conformity, which is amenable to static
condensation and also makes for a simple and effective mass-matrix
based approximation to the trace pressure Schur complement.

Let us mention that alternative AL approaches for the HDG method for
Navier--Stokes have been considered. In \cite{Southworth:2020}, during
the static condensation step, we only eliminated cell velocity
degrees-of-freedom (dofs) instead of cell velocity and cell pressure
dofs as we do in the current paper. In this approach the velocity is
not exactly divergence-free at intermediate iterations and so a
standard grad-div penalty term could be used. However, a large number
of iterations were required to reach convergence (see also
\cite{Sivas:thesis}). An augmented Lagrangian Uzawa iteration method
was proposed in \cite{Fu:2024} for a different HDG method that was
introduced in \cite{Cockburn:2014b,Fu:2019}. They show good results
with only a small dependence on the Reynolds number for $Re<1000$ in
2D and, after adding a pseudo-time integration method, also in 3D.

The outline of this paper is as follows. In \cref{sec:hdgns} we
present the HDG discretization for the Navier--Stokes equations. We
present new matrix-based theory on preconditioning saddle-point
problems with large singular perturbations in \cref{sec:precon}, and
use this to motivate our proposed AL block preconditioning for the HDG
method.  Although AL formulations and preconditioning typically allow
a simple and effective approximation to the Schur complement, much of
the difficulty is transferred to solving the augmented velocity block,
which takes the form of an advection-dominated vector
advection-diffusion equation with a large symmetric singular
perturbation. Here we recognize certain recent developments in the
STRUMPACK library \cite{Claus:2023,Liu:2021} are well-suited to
solving such problems, and are also algebraic (not requiring the
infrastructure of geometric multigrid methods), a particularly useful
property in the context of HDG discretizations. We thus combine the
STRUMPACK solver for the augmented velocity block with our AL
formulation and preconditioning of the coupled system, for a complete
method.  In \cref{sec:num_examples} we demonstrate our proposed
framework on steady lid-driven cavity and backward-facing step
problems up to a Reynolds number of $10^4$. The outer AL-based block
preconditioner is demonstrated to be highly robust in mesh spacing and
Reynolds number, and the inner STRUMPACK solver performs well, only
degrading at very high Reynolds numbers, where the steady state
solution becomes very ill-conditioned. However, as discussed in
\cref{sec:num_examples}, we expect the algebraic preconditioner
proposed here to function effectively for any realistic CFD simulation
conditions. We conclude in \cref{sec:conclusions}.

\section{HDG for the Navier--Stokes equations}
\label{sec:hdgns}

\subsection{Notation and finite element spaces}

Let $\mathcal{T}_h = \cbr[0]{K}$ be a simplicial mesh consisting of
nonoverlapping cells $K$ and such that
$\cup_{K \in \mathcal{T}_h} K = \overline{\Omega}$. Consider two
adjacent cells $K^+$ and $K^-$. An interior face is defined as
$F := \partial K^+ \cap \partial K^-$. A boundary face is defined as a
face of cell $K$ that lies on the boundary. We will denote the set of
all faces by $\mathcal{F}_h$ and the union of all faces by
$\Gamma^0$. The diameter of a cell is denoted by $h_K$ and denote the
outward unit normal vector on $\partial K$ by $n$. Furthermore, for
scalar functions $u$ and $v$, we will write
$(u,v)_{K} = \int_{K}uv \dif x$,
$(u,v)_{\mathcal{T}} = \sum_{K \in \mathcal{T}_h} (u,v)_K$,
$\langle u,v \rangle_{\partial K} = \int_{\partial K} uv \dif s$,
$\langle u,v \rangle_{\partial \mathcal{T}} = \sum_{K \in
  \mathcal{T}_h} \langle u, v\rangle_{\partial K}$,
$\langle u, v \rangle_F = \int_F uv \dif s$, and
$\langle u, v\rangle_{\mathcal{F}} = \sum_{F \in \mathcal{F}_h}
\langle u, v\rangle_F$. Similar notation is used if $u,v$ are vector
or matrix functions. We denote the trace of a vector function $v$ from
the interior of $K^{\pm}$ by $v^{\pm}$ and the outward unit normal to
$K^{\pm}$ by $n^{\pm}$. We then define the jump and average operators
on an interior face as
$\jump{v \cdot n} := v^+\cdot n^+ + v^-\cdot n^-$ and
$\av{v} := \tfrac{1}{2}(v^+ + v^-)$, respectively. On a boundary face
the jump operator is defined as $\jump{v \cdot n} := v \cdot n$ while
the average operator is $\av{v} := 0$.

The velocity on $\Omega$ and trace of the velocity on $\Gamma^0$ are
approximated by functions in the following finite element spaces:
\begin{subequations}
  \begin{align}
    \label{eq:femVh}
    V_h
    &:= \cbr[1]{v_h\in L^2(\Omega)^d
      : \ v_h \in P_k(K)^d, \ \forall\ K\in\mathcal{T}},
    \\
    \label{eq:femVbh}
    \bar{V}_h
    &:= \cbr[1]{\bar{v}_h \in L^2(\Gamma^0)^d:\ \bar{v}_h \in
      P_{k}(F)^d\ \forall\ F \in \mathcal{F},\ \bar{v}_h
      = 0 \text{ on } \partial \Omega},
  \end{align}
\end{subequations}
where $P_k(K)$ and $P_k(F)$ are the sets of polynomials of degree at
most $k \ge 1$ defined on a cell $K \in \mathcal{T}_h$ and face
$F \in \mathcal{F}_h$. Similarly, the pressure on $\Omega$ and trace
of the pressure on $\Gamma^0$ are approximated by functions in
\begin{align*}
  Q_h
  &:= \cbr[1]{q_h\in L^2(\Omega) : \ q_h \in P_{k-1}(K) ,\
    \forall \ K \in \mathcal{T}},
  \\
  \bar{Q}_h
  &:= \cbr[1]{\bar{q}_h \in L^2(\Gamma^0) : \ \bar{q}_h \in
    P_{k}(F) \ \forall\ F \in \mathcal{F}}.
\end{align*}
Product spaces and pairs in these product spaces will be denoted using
boldface. For example,
$\boldsymbol{v}_h := (v_h, \bar{v}_h) \in V_h \times \bar{V}_h =:
\boldsymbol{V}_h$ and
$\boldsymbol{q}_h := (q_h, \bar{v}_h) \in Q_h \times \bar{Q}_h =:
\boldsymbol{Q}_h$.

\subsection{The discretization}

We consider the following HDG discretization of the Navier--Stokes
equations \cite{Rhebergen:2018a}: Find
$(\boldsymbol{u}_h, \boldsymbol{p}_h) \in \boldsymbol{V}_h \times
\boldsymbol{Q}_h$ such that for all
$(\boldsymbol{v}_h, \boldsymbol{q}_h) \in \boldsymbol{V}_h \times
\boldsymbol{Q}_h$:
\begin{subequations}
  \label{eq:discrete_problem}
  \begin{align}
    \label{eq:discrete_problem_a}
    \mu a_h(\boldsymbol{u}_h, \boldsymbol{v}_h)
    + o_h(u_h; \boldsymbol{u}_h, \boldsymbol{v}_h)
    + b_h(\boldsymbol{p}_h, v_h)
    &=
      (f, v)_{\mathcal{T}},
    \\
    \label{eq:discrete_problem_b}
    b_h(\boldsymbol{q}_h, u_h)
    &= 0,
  \end{align}
\end{subequations}
where $a_h(\cdot, \cdot)$ and $o_h(\cdot;\cdot, \cdot)$ are defined as
\begin{align*}
  a_h(\boldsymbol{u}, \boldsymbol{v})
  :=& (2 \varepsilon(u), \varepsilon(v))_{\mathcal{T}}
      + \langle 2 \alpha h_K^{-1}(u-\bar{u}), v-\bar{v} \rangle_{\partial\mathcal{T}}
  \\
    &  - \langle 2\varepsilon(u)n, v-\bar{v} \rangle_{\partial\mathcal{T}}
       - \langle 2\varepsilon(v)n, u-\bar{u} \rangle_{\partial\mathcal{T}},
  \\
  o_h(w; \boldsymbol{u}, \boldsymbol{v})
  :=& - (u\otimes w, \nabla v)_{\mathcal{T}}
      + \tfrac{1}{2} \langle (w\cdot n) (u+\bar{u}) + \envert{w\cdot n}(u-\bar{u}), v-\bar{v} \rangle_{\partial \mathcal{T}},
\end{align*}
and where $b_h(\boldsymbol{q}, v) := b_1(q,v) + b_2(\bar{q},v)$ in
which
\begin{align*}
  b_1(q,v) &:= - (q, \nabla \cdot v)_{\mathcal{T}},
  &
  b_2(\bar{q},v) &:= \langle \bar{q},  v \cdot n \rangle_{\partial\mathcal{T}}.
\end{align*}
The interior penalty parameter $\alpha > 0$ in the definition of
$a_h(\cdot, \cdot)$ needs to be chosen sufficiently large for the HDG
method to be stable (see \cite{Rhebergen:2017}). The HDG
discretization \cref{eq:discrete_problem} was shown in
\cite{Rhebergen:2018a} to result in a discrete velocity solution
$u_h \in V_h$ that is exactly divergence-free on the cells $K$ and
globally $H(\text{div};\Omega)$-conforming. Consequently, the analysis
in \cite{Kirk:2019} shows that this discretization is pressure-robust
\cite{John:2017}.

\subsection{The algebraic formulation}

To solve the nonlinear problem \cref{eq:discrete_problem} we use
Picard or Newton iterations. For the sake of presentation we consider
a Picard iteration here, but numerical examples in
\cref{sec:num_examples} use Newton iterations.

Let $u \in \mathbb{R}^{n_u}$, $\bar{u} \in \mathbb{R}^{\bar{n}_u}$,
$p \in \mathbb{R}^{n_p}_0 := \cbr{q \in \mathbb{R}^{n_p}|q \ne 1}$ and
$\bar{p} \in \mathbb{R}^{\bar{n}_p}$ represent the vectors of degrees
of freedom for $u_h$, $\bar{u}_h$, $p_h$, and $\bar{p}_h$,
respectively. Furthermore, let
$\mathbb{V} := \cbr[0]{ {\bf v} = [v^T\, \bar{v}^T]^T\, : v \in
  \mathbb{R}^{n_u},\ \bar{v} \in \mathbb{R}^{\bar{n}_u}}$ and
$\mathbb{Q} := \cbr[0]{ {\bf q} = [q^T\, \bar{q}^T]^T\, : q \in
  \mathbb{R}^{n_q}_0,\ \bar{q} \in \mathbb{R}^{\bar{n}_q}}$. The forms
$a_h(\cdot, \cdot)$, $o_h(w;\cdot, \cdot)$ for $w$ given,
$b_1(\cdot, \cdot)$, and $b_2(\cdot, \cdot)$ are expressed as matrices
as follows:
\begin{equation}
  \label{eq:def-matrices}
  \begin{aligned}
    a_h(\boldsymbol{v}_h, \boldsymbol{v}_h)
    &
    = [v^T, \bar{v}^T]
    \begin{bmatrix}
      A_{uu} & A_{\bar{u}u}^T \\
      A_{\bar{u}u} & A_{\bar{u}\bar{u}}
    \end{bmatrix}
    \begin{bmatrix}
      v \\ \bar{v}
    \end{bmatrix}
    && \forall {\bf v} \in \mathbb{V},
    \\
    o_h(w;\boldsymbol{v}_h, \boldsymbol{v}_h)
    &
    = [v^T, \bar{v}^T]
    \begin{bmatrix}
      N_{uu} & N_{u\bar{u}} \\
      N_{\bar{u}u} & N_{\bar{u}\bar{u}}
    \end{bmatrix}
    \begin{bmatrix}
      v \\ \bar{v}
    \end{bmatrix}
    && \forall {\bf v} \in \mathbb{V},
    \\
    b_1(q_h, v_h)
    &
    = q^T B_{pu}v
    && \forall q \in \mathbb{R}_0^{n_p},\ v \in \mathbb{R}^{n_u},
    \\
    b_2(\bar{q}_h, v_h)
    &
    = \bar{q}^T B_{\bar{p}u} v
    && \forall \bar{q}^T \in \mathbb{R}^{\bar{n}_q},\ v \in \mathbb{R}^{n_u}.
  \end{aligned}
\end{equation}
Defining also
\begin{equation}
  \label{eq:FANDJ}
  F =
  \begin{bmatrix}
    F_{uu} & F_{u\bar{u}} \\
    F_{\bar{u}u} & F_{\bar{u}\bar{u}}
  \end{bmatrix}
  =
  \begin{bmatrix}
    \mu A_{uu} + N_{uu} & \mu A_{\bar{u}u}^T + N_{u\bar{u}} \\
    \mu A_{\bar{u}u} + N_{\bar{u}u}   & \mu A_{\bar{u}\bar{u}} + N_{\bar{u}\bar{u}}
  \end{bmatrix},
\end{equation}
the algebraic form of the linearized HDG method that needs to be
solved at each Picard iteration can be written as:
\begin{equation}
  \label{eq:sclinsystemAL-reor}
  \begin{bmatrix}
    F_{uu} & B_{pu}^T & F_{u\bar{u}} & B_{\bar{p}u}^T \\
    B_{pu} & 0 & 0 & 0 \\
    F_{\bar{u}u} & 0 & F_{\bar{u}\bar{u}} & 0 \\
    B_{\bar{p}u} & 0 & 0 & 0
  \end{bmatrix}
  \begin{bmatrix}
    u \\ p \\ \bar{u} \\ \bar{p}
  \end{bmatrix}
  =
  \begin{bmatrix}
    L_u \\
    0 \\ 0 \\ 0
  \end{bmatrix},
\end{equation}
where $L_u$ is the vector representation of the right hand side of
\cref{eq:discrete_problem_a}.

Since the matrices $F_{uu}$ and $B_{pu}$ are block diagonal with one
block per cell, $u$ and $p$ can be eliminated locally via static
condensation.

\section{Preconditioning}
\label{sec:precon}

In this section we present preconditioning for the linearized HDG
method for Navier--Stokes. However, instead of preconditioning
\cref{eq:sclinsystemAL-reor} directly, we consider a modified version
of \cref{eq:sclinsystemAL-reor} that is easier to precondition. For
this, similar to the Augmented Lagrangian approach of
\cite{Benzi:2006}, we augment the original problem
\cref{eq:sclinsystemAL-reor} with a suitable physically consistent
penalty term that makes the Schur complement easier to approximate
(see \cref{ss:preconditioningHDG}).

We preface this section by recalling that for solving nonsingular
$2\times 2$ block linear systems
\begin{equation}
  \label{eq:2x2}
  \begin{bmatrix}
    A & B \\
    C & D
  \end{bmatrix}
  \begin{bmatrix}
    \mathbf{x} \\
    \mathbf{y}
  \end{bmatrix}
  =
  \begin{bmatrix}
    \mathbf{f} \\
    \mathbf{g}
  \end{bmatrix},
\end{equation}
a standard approach is to apply a block triangular preconditioner,
where one diagonal block is (approximately) inverted and the
complementary Schur complement is approximated. For example, a block
lower triangular preconditioner would take the form
\begin{equation}
  \label{eq:triS}
  \mathbb{P} =
  \begin{bmatrix}
    A & 0
    \\ C & \widehat{S}
  \end{bmatrix}^{-1},
\end{equation}
where
\begin{equation}
  \label{eq:S}
  \widehat{S} \approx S := D - CA^{-1}B.
\end{equation}
For approximate block triangular or block LDU preconditioners of this
form, convergence of fixed-point and minimal residual Krylov methods
is fully defined by the convergence of an equivalent method applied to
$S$, preconditioned with $\widehat{S}$. For more details, see
\cite{Southworth:2020}. The important point, which we use as the
primary objective in designing our preconditioner, is that the key
component to effective preconditioning of \cref{eq:2x2} is an accurate
and computable approximation $\widehat{S}^{-1}\approx S^{-1}$.

In \cref{ss:singularperturb} we first present some general results on
inverse singular perturbations. We then proceed in
\cref{ss:preconditioningHDG} to formulate a preconditioner for the HDG
method. In this section we will denote the nullspace and range of a
matrix $A$ by $\mathcal{N}(A)$ and $\mathcal{R}(A)$, respectively.

\subsection{On the inverse of singular perturbations}
\label{ss:singularperturb}

AL-preconditioning is built around adding a symmetric singular
perturbation to the leading block in a $2\times 2$ system, which makes
the Schur complement easier to approximate. Here we build on linear
algebra theory from \cite[Lemma 10]{Haut:2020}, and derive an inverse
expansion for the Schur complement of augmented saddle-point systems
with respect to perturbation parameter $\gamma \gg 0$.

\begin{lemma}[Singular perturbation of saddle-point systems]
  \label{lem:singularMatrixPerturbation}
  Let $A$ be a nonsingular $n \times n$ matrix, $J \ne 0$ be a
  symmetric, singular $n \times n$ matrix, and $B$ be an $m \times n$
  matrix ($m < n$) such that $\mathcal{N}(J)\subseteq \mathcal{N}(B)$
  and $\mathcal{R}(B^T) = \mathcal{R}(J)$. Let $\gamma>0$ be a scalar
  constant and assume the following matrix is nonsingular:
  \begin{equation*}
    \mathbb{A} =
    \begin{bmatrix}
      A + \gamma J & B^T \\
      B & \mathbf{0}
    \end{bmatrix}.
  \end{equation*}
  Let $P$ be an orthogonal projection onto $\mathcal{N}(J)$, let
  $Q=I-P$ denote its complement, and define
  \begin{equation*}
    E_{P} = P \del[0]{PAP}^{-1} P,
    \quad
    E_{Q} = Q \del[0]{ QJQ }^{-1} Q,
  \end{equation*}
  as projected inverses over the range of $P$ and $Q$. Let
  $S = -B(A+\gamma J)^{-1}B^T$ denote the Schur complement of
  $\mathbb{A}$. The action of $S$ takes the form
  \begin{equation}
    \label{eq:S-exp-lemma}
    -S\mathbf{x}
    = \gamma^{-1} BE_QB^T\mathbf{x} +  \gamma^{-2} BR_{\gamma} B^T\mathbf{x},
  \end{equation}
  where
  \begin{equation}
    \label{eq:Rgamdef}
    R_{\gamma}
    = \del[1]{ I+\gamma^{-1} E_{Q}A(I - E_PA) }^{-1} E_{Q}A(I - E_PA)E_Q,
  \end{equation}
  and $E_PA$ is a projection onto $\mathcal{N}(J)$.
\end{lemma}
\begin{proof}
  A rescaling of the result from \cite[Lemma 10]{Haut:2020} on an
  inverse expansion for singular perturbations yields the
  following:\footnote{ Note, there is a typo in the statement of
    \cite[Lemma 10]{Haut:2020}; $R_\varepsilon$ as occurs in
    \cite[Eq. (34)]{Haut:2020} should not have a leading factor of
    $\varepsilon$; this constant was already accounted for in the
    inverse expansion in the third term of
    \cite[Eq. (34)]{Haut:2020}.}
  \begin{equation*}
    (A + \gamma J)^{-1}\mathbf{x}
    =
    E_{P}\mathbf{x}
    + \frac{1}{\gamma}\del[0]{I-E_{P}A}E_{Q}\del[0]{I-AE_{P}}\mathbf{x}
    + \frac{1}{\gamma^2}\del[0]{I-E_{P}A}R_{\gamma} \del[0]{I-AE_{P}}\mathbf{x},
  \end{equation*}
  where
  \begin{equation}
    R_{\gamma} = \del[1]{ I+\gamma^{-1} E_{Q} \del[0]{A-AE_{P}A} Q }^{-1}E_{Q} A\del[0]{ I-E_{P}A }E_Q.
  \end{equation}
  Recalling $Q = I-P$,
  \begin{align*}
    \del[0]{A-AE_{P}A}Q
    & = A(I - P(PAP)^{-1}PA)(I - P) \\
    & = A(I - P(PAP)^{-1}PA) \\
    & = A(I - E_PA),
  \end{align*}
  which yields \cref{eq:Rgamdef}.  Then, note that since $P$ projects
  onto $\mathcal{N}(J)\subseteq\mathcal{N}(B)$, we have
  $BP = PB^T = \mathbf{0}$.  This yields
  \begin{align*}
    -S\mathbf{x}
    & = B(A + \gamma J)^{-1}B^T\mathbf{x}
    \\
    & = B\sbr[2]{E_{P}+ \gamma^{-1}\del[0]{I-E_{P}A}E_{Q}\del[0]{I-AE_{P}}
      +\gamma^{-2}\del[0]{I-E_{P}A}R_{\gamma}\del[0]{I-AE_{P}} } B^T\mathbf{x}
    \\
    & = \gamma^{-1} BE_QB^T\mathbf{x} + \gamma^{-2} BR_{\gamma} B^T\mathbf{x},
  \end{align*}
  which is the desired result.
\end{proof}

\Cref{lem:singularMatrixPerturbation} provides an inverse expansion
for arbitrary symmetric singular perturbation $J$ with nullspace
contained within that of $B$. Often in saddle-point problems and
AL-type preconditioners, one chooses $J$ of the specific form
$J = B^TM^{-1}B$ for mass matrix $M$. Such a choice also leads to a
further simplification of the action of Schur complement $S$, which is
derived in the following corollary.

\begin{corollary}[Augmented Lagrangian Schur complements]
  \label{cor:2x2}
  Let $A$, $B$, and $\gamma$ be as in
  \cref{lem:singularMatrixPerturbation}, and let $M$ be an
  $m \times m$ symmetric positive definite (SPD) matrix. Consider the
  matrix
  \begin{equation}
    \label{eq:2x2-AL}
    \mathbb{A} =
    \begin{bmatrix}
      A + \gamma B^TM^{-1}B & B^T
      \\
      B & \mathbf{0}
    \end{bmatrix}.
  \end{equation}
  The action of the Schur complement $S$ of $\mathbb{A}$ takes the
  form
  \begin{equation*}
    -S\mathbf{x}
    = \gamma^{-1} M\mathbf{x}
    + \gamma^{-2} BR_\gamma B^T\mathbf{x}.
  \end{equation*}
\end{corollary}
\begin{proof}
  Let $P$ be the $\ell^2$-orthogonal projection on to the nullspace of
  $B^TM^{-1}B$, given by
  \begin{equation*}
    P
    := (I - B^TM^{-1/2}(M^{-1/2}BB^TM^{-1/2})^{-1}M^{-1/2}B)
    = (I - B^T(BB^T)^{-1}B),
  \end{equation*}
  and let
  \begin{equation*}
    Q
    := I-P
    = B^TM^{-1/2}(M^{-1/2}BB^TM^{-1/2})^{-1}M^{-1/2}B
    = B^T(BB^T)^{-1}B,
  \end{equation*}
  denote its complement. Then observe that
  \begin{align*}
    & M^{-1/2}B(QB^TM^{-1}BQ)^{-1}B^TM^{-1/2}\mathbf{x}
    \\
    & = M^{-1/2}B\Big[B^TM^{-1/2}(M^{-1/2}BB^TM^{-1/2})^{-1}M^{-1/2}BB^TM^{-1/2}
    \\ & \hspace{15ex} M^{-1/2}BB^TM^{-1/2}(M^{-1/2}BB^TM^{-1/2})^{-1}M^{-1/2}B\Big]^{-1}B^TM^{-1/2}\mathbf{x}
    \\
    & = M^{-1/2}B\del[1]{ B^TM^{-1/2}M^{-1/2}B }^{-1}B^TM^{-1/2}\mathbf{x}.
  \end{align*}
  To analyze this final term, let $B^TM^{-1/2}$ have SVD
  $B^TM^{-1/2} = L\Sigma R^T$, for orthogonal left and right singular
  vectors as columns of $L\in\mathbb{R}^{n\times n}$ and
  $R\in\mathbb{R}^{m\times m}$, and singular values as nonzero
  diagonal of $\Sigma\in\mathbb{R}^{n\times m}$. Let
  $\Sigma^T\Sigma\in\mathbb{R}^{m\times m}$ with positive diagonal
  entries and $\Sigma\Sigma^T\in\mathbb{R}^{n\times n}$ have $m$
  positive diagonal entries and $n-m$ zeros corresponding to the
  nullspace. Then
  \begin{align*}
    M^{-1/2}B(QB^TM^{-1}BQ)^{-1}B^TM^{-1/2}\mathbf{x}
    & = R\Sigma^TL^T\del[1]{L\Sigma R^TR\Sigma^TL^T}^{-1}L\Sigma R^T\mathbf{x}
    \\
    & = R\Sigma^T\del[1]{\Sigma \Sigma^T}^{-1}\Sigma R^T\mathbf{x}.
  \end{align*}
  Note we have changed the basis via the SVD, and here the inverse
  implies we are inverting on the space of nonzero singular values,
  where $\Sigma\Sigma^T\in\mathbb{R}^{n\times n}$ is nonzero in the
  leading $m$ diagonal entries and zero in the final $n-m$ entries. In
  this setting, this is trivially equivalent to a pseudoinverse. The
  result is that
  $\Sigma^T \del[0]{ \Sigma \Sigma^T }^{-1}\Sigma = I_m$, implying
  \begin{equation*}
    B(QB^TM^{-1}BQ)^{-1}B^TM^{-1/2}\mathbf{x} = M^{1/2}\mathbf{x}
    \quad \Longleftrightarrow \quad
    B(QB^TM^{-1}BQ)^{-1}B^T\mathbf{y} = M\mathbf{y},
  \end{equation*}
  for any $\mathbf{x} = M^{1/2}\mathbf{y}$. Given the assumption that
  $M$ is SPD, this applies for all vectors
  $\mathbf{y}\in\mathbb{R}^m$. Plugging into \cref{eq:S-exp-lemma}
  completes the proof.
\end{proof}

Thus to order $1/\gamma$, the Schur complement of \cref{eq:2x2-AL} can
be preconditioned by $\gamma^{-1}M$. It should be noted that this
approximation may not be $h$-independent, in the sense that the
residual term $R_{\gamma}$ can have a complex
eigenvalue/field-of-value structure that may depend on mesh-spacing
$h$. To this end, robust preconditioning may require relatively large
$\gamma$.

\subsection{Preconditioning for HDG}
\label{ss:preconditioningHDG}

The results of \cref{ss:singularperturb} are now used to design
preconditioners for the linearized form of the HDG discretization
\cref{eq:discrete_problem} augmented with a suitable penalization
term. We begin with a traditional AL-like preconditioning along the
lines of $J = BM^{-1}B^T$ (see \Cref{cor:2x2}) in
\Cref{ss:preconditioningHDG:duu}. However, such an approach does not
facilitate static condensation, so we then use the more general result
in \Cref{lem:singularMatrixPerturbation} to motivate an approximate
preconditioner by focusing on a local representation of the nullspace
of $B$.

\subsubsection{AL-like preconditioning}
\label{ss:preconditioningHDG:duu}

We begin with an augmentation following a traditional AL-approch for
incompressible Navier Stokes. Rather than a grad-div augmentation
though, define (see, for example, \cite{Akbas:2018})
\begin{equation}
  \label{eq:dhu}
  d_h(u,v) = \langle h_F^{-1} \jump{u \cdot n}, \jump{v \cdot n}
  \rangle_{\mathcal{F}},
\end{equation}
and let $D_{uu} \in \mathbb{R}^{n_u \times n_u}$ be the matrix such
that $d_h(v_h, v_h) = v^TD_{uu}v$.  The following Lemma proves that
$d_h(u,v)$ is the weak form associated with an augmentation
$D_{uu} = B_{\bar{p}u}^T\bar{M}^{-1}B_{\bar{p}u}$.

\begin{lemma}
  \label{lem:factDuu}
  Let $\bar{M}$ be the trace pressure mass matrix defined by
  $\norm[0]{\bar{q}_h}_p^2 := \bar{q}^T\bar{M} \bar{q}$, where
  $\norm[0]{\bar{q}_h}_p^2 := \sum_{F \in \mathcal{F}_h} h_F
  \norm[0]{\bar{q}_h}_{F}^2$ for all $\bar{q}_h \in \bar{Q}_h$. It
  holds that $D_{uu} = B_{\bar{p}u}^T\bar{M}^{-1}B_{\bar{p}u}$.
\end{lemma}
\begin{proof}
  See Appendix \ref{app:proofs}.
\end{proof}

Defining
\begin{equation*}
  A :=
  \begin{bmatrix}
    F_{uu} & B_{pu}^T & F_{u\bar{u}} \\
    B_{pu} & 0 & 0  \\
    F_{\bar{u}u} & 0 & F_{\bar{u}\bar{u}}
  \end{bmatrix},
\end{equation*}
as the matrix over $\cbr[0]{u, p, \bar{u}}$, and
$B=\begin{bmatrix} B_{\bar{p}u} & 0 & 0 \end{bmatrix}$, \cref{cor:2x2}
together with \cref{lem:factDuu} suggest the following (block
lower-triangular) preconditioner
\begin{equation}
  \label{eq:precon_global}
  \mathbb{P}_D =
  \begin{bmatrix}
    F_{uu} + \gamma D_{uu} & B_{pu}^T & F_{u\bar{u}} & 0 \\
    B_{pu} & 0 & 0 & 0 \\
    F_{\bar{u}u} & 0 & F_{\bar{u}\bar{u}} & 0 \\
    B_{\bar{p}u} & 0 & 0 & -\gamma^{-1} \bar{M}
  \end{bmatrix},
\end{equation}
as an augmented block preconditioner for the following penalized HDG
discretization:
\begin{subequations}
  \label{eq:discrete_problem_ALg}
  \begin{align}
    \label{eq:discrete_problem_ALg_a}
    \mu a_h(\boldsymbol{u}_h, \boldsymbol{v}_h)
    + o_h(u_h; \boldsymbol{u}_h, \boldsymbol{v}_h)
    + \gamma d_h(u_h, v_h)
    + b_h(\boldsymbol{p}_h, v_h)
    &=
      (f, v_h)_{\mathcal{T}},
    \\
    \label{eq:discrete_problem_ALg_b}
    b_h(\boldsymbol{q}_h, u_h)
    &= 0.
  \end{align}
\end{subequations}
In theory, if we eliminate $u$ and $p$ for a reduced system defined
only on faces, we arrive at the following condensed preconditioner:
\begin{equation}
  \label{eq:condensed2form_up_pre}
  \overline{\mathbb{P}}_D
  =
  \begin{bmatrix}
    \bar{F} &
    0 \\
    -B_{\bar{p}u}\mathcal{P}(F_{uu}+\gamma D_{uu})^{-1}F_{u\bar{u}} &
    -\gamma^{-1}\bar{M} - B_{\bar{p}u}\mathcal{P}(F_{uu}+\gamma D_{uu})^{-1}B_{\bar{p}u}^T
  \end{bmatrix},
\end{equation}
where
\begin{equation*}
  \begin{split}
    \bar{F}
    &= -F_{\bar{u}u}\mathcal{P}(F_{uu}+\gamma D_{uu})^{-1}F_{u\bar{u}} + F_{\bar{u}\bar{u}},
    \\
    \mathcal{P}
    &= I - (F_{uu} + \gamma D_{uu})^{-1}B_{pu}^T(B_{pu}(F_{uu}+\gamma D_{uu})^{-1}B_{pu}^T)^{-1}B_{pu}.
  \end{split}
\end{equation*}

Above, we proposed a block preconditioner first and then statically
condense the $4\times4$ block system to a $2\times 2$ block system
over faces. Alternatively, we can first condense to faces and then
define our preconditioner. Recall the discussion from the beginning of
this section regarding convergence of fixed-point and Krylov methods
being fully defined by the approximation to the Schur complement. Note
that the Schur complement in $\bar{p}$ is identical regardless of
whether it is obtained by eliminating the full $\{u,\bar{u},p\}$
block, or first by condensing to a reduced equation in
$\{\bar{u},\bar{p}\}$ and then forming a Schur complement in
$\bar{p}$. To that end, one can also take \cref{cor:2x2} and
\cref{lem:factDuu} together to imply $ -\gamma^{-1}\bar{M}$ is a good
approximation to the Schur complement in $\bar{p}$ and apply this to
the reduced system, resulting in preconditioner
\begin{equation}
  \label{eq:condensed2form_up_pre-onlyM}
  \overline{\mathbb{P}}_{D,M}
  =
  \begin{bmatrix}
    \bar{F} & 0
    \\
    -B_{\bar{p}u}\mathcal{P}(F_{uu}+\gamma D_{uu})^{-1}F_{u\bar{u}} &
    -\gamma^{-1}\bar{M}
  \end{bmatrix}.
\end{equation}
Note the only difference with $\overline{\mathbb{P}}_D$ in
\cref{eq:condensed2form_up_pre} is that here we only include the mass
matrix in the approximate Schur complement, whereas in
\cref{eq:condensed2form_up_pre} we also include the condensed
$\bar{p}\bar{p}$-block of the system as well.

\subsubsection{Preconditioner for static condensation}
\label{ss:preconditioningHDG:hdg}

Unfortunately, it is expensive to apply static condensation to the HDG
method \cref{eq:discrete_problem_ALg} because $d_h(\cdot, \cdot)$
couples velocity cell degrees of freedom across multiple cells. As a
result, the inverse of $F_{uu}+\gamma D_{uu}$ is typically dense, and
we cannot directly construct the reduced system in
$\{\bar{u},\bar{p}\}$.  To that end, we seek a local augmentation that
maintains the ease of approximating $S$ in $\bar{p}$, but also
facilitates static condensation.  We motivate this by considering the
more general result on the inverse of singular perturbations,
\Cref{lem:singularMatrixPerturbation}. There we see that the key
assumption in choosing an augmentation matrix $J$ is that
$\mathcal{N}(J) \subseteq\mathcal{N}(B)$. In
\Cref{ss:preconditioningHDG:duu} we choose $d_h(u,v)$ \cref{eq:dhu}
and the assembled matrix $D_{uu}$ such that
$\mathcal{N}(B_{\bar{p}u}) = \mathcal{N}(D_{uu}) = V_h \cap
H(\text{div};\Omega)$ and so also $\mathcal{N}(B) = \mathcal{N}(J)$.

A hybridizable alternative to $d_h(u,v)$, chosen such that the
nullspace is a subspace of the nullspace of $d_h(u,v)$, is given by
\begin{equation}
  \label{eq:ghu}
  g_h(\boldsymbol{u}, \boldsymbol{v}) := \langle
  h_K^{-1}(u-\bar{u})\cdot n, (v-\bar{v})\cdot n
  \rangle_{\partial\mathcal{T}},
\end{equation}
which in matrix form is expressed as
\begin{equation*}
  g_h(\boldsymbol{v}_h, \boldsymbol{v}_h)
  = [v^T, \bar{v}^T] G
  \begin{bmatrix}
    v \\ \bar{v}
  \end{bmatrix}
  = [v^T, \bar{v}^T]
  \begin{bmatrix}
    G_{uu} & G_{\bar{u}u}^T \\
    G_{\bar{u}u} & G_{\bar{u}\bar{u}}
  \end{bmatrix}
  \begin{bmatrix}
    v \\ \bar{v}
  \end{bmatrix}
  \qquad \forall {\bf v} \in \mathbb{V}.
\end{equation*}
Let
$R_k(\partial K) := \cbr[0]{\mu \in L^2(\partial K)\,:\, \mu|_F \in
  P_k(F),\ \forall F \subset \partial K}$ and let
$b_2^K(\bar{q}_h, \boldsymbol{u}_h) = \langle (u_h - \bar{u}_h) \cdot
n, \bar{q}_h \rangle_{\partial K}$. Then the nullspace of $G$ is given
by
\begin{equation}
  \mathcal{N}(G) :=
  \cbr[0]{ \boldsymbol{v}_h \in \boldsymbol{V}_h \, :\,
    b_2^K(\bar{q}_h, \boldsymbol{u}_h) = 0\ \forall \bar{q}_h \in R_k(\partial K),\ \forall K \in \mathcal{T}_h }.
\end{equation}
Note that
\begin{equation}
  \label{eq:globalHdivenforce}
  \begin{split}
    b_2(\bar{q}_h, u_h) &= \sum_{K \in \mathcal{T}_h} \langle
    u_h \cdot n, \bar{q}_h \rangle_{\partial K} + \langle u_h\cdot
    n, \bar{q}_h \rangle_{\partial \Omega}
    \\
    &= \sum_{K \in \mathcal{T}_h} \langle (u_h - \bar{u}_h) \cdot n,
    \bar{q}_h \rangle_{\partial K} = \sum_{K \in \mathcal{T}_h}
    b_2^K(\bar{q}_h, \boldsymbol{u}_h),
  \end{split}
\end{equation}
where the second equality is because $\bar{u}_h \cdot n$ and
$\bar{q}_h$ are single-valued on interior faces and $\bar{u}_h = 0$ on
the boundary. This is used to show the following relationship between
$\mathcal{N}(G)$ and $\mathcal{N}(D_{uu})$:
\begin{equation}
  \label{eq:NGvsND}
  \begin{split}
    \mathcal{N}(G)
    &\subset
    \cbr[2]{\boldsymbol{v}_h \in \boldsymbol{V}_h\, :\, \sum_{K \in \mathcal{T}_h}b_2^K(\bar{q}_h, v_h) = 0 \ \forall \bar{q}_h \in \bar{Q}_h}
    \\
    &=
    \cbr[0]{\boldsymbol{v}_h \in \boldsymbol{V}_h\, :\, b_2(\bar{q}_h, v_h) = 0 \ \forall \bar{q}_h \in \bar{Q}_h}
    \\
    &= \cbr[0]{\boldsymbol{v}_h \in \boldsymbol{V}_h\, :\, v_h \in V_h \cap H(\text{div},\Omega)}
    \\
    &= \sbr[1]{V_h \cap H(\text{div},\Omega)} \times \bar{V}_h
    \\
    &= \mathcal{N}(D_{uu}) \times \bar{V}_h.
  \end{split}
\end{equation}
In other words, the nullspace of $G$ is smaller than that of $D_{uu}$
as $\boldsymbol{v}_h \in \mathcal{N}(G)$ requires
$v_h \cdot n = \bar{v}_h \cdot n$ on all faces in addition to
$v_h \in V_h \cap H(\text{div};\Omega)$.

By nature of $\mathcal{N}(G)\subset \mathcal{N}(D_{uu})$,
\Cref{lem:singularMatrixPerturbation} still naturally applies.  In
\Cref{cor:2x2} and \cref{ss:preconditioningHDG:duu} we are able to
express the leading order term in $\mathcal{S}$ specifically as a mass
matrix. Here we no longer have the analytical representation of the
leading term in $\gamma$, but posit that the mass matrix remains a
good approximation to the leading order term in $\gamma$, $BE_QB^T$, a
hypothesis confirmed in experiments in \cref{sec:num_examples}.
Replacing the augmentation $D_{uu}$ in \cref{eq:precon_global} with
the matrix form of $g_h(\cdot, \cdot)$, we find the following block
preconditioner
\begin{equation}
  \label{eq:precon_sc}
  \mathbb{P}_G =
  \begin{bmatrix}
    F_{uu} + \gamma G_{uu} & B_{pu}^T & F_{u\bar{u}} + \gamma G_{\bar{u}u}^T & 0 \\
    B_{pu} & 0 & 0 & 0 \\
    F_{\bar{u}u} + \gamma G_{\bar{u}u} & 0 & F_{\bar{u}\bar{u}} + \gamma G_{\bar{u}\bar{u}} & 0 \\
    B_{\bar{p}u} & 0 & 0 & -\gamma^{-1} \bar{M}
  \end{bmatrix},
\end{equation}
as an augmented block preconditioner for the following penalized HDG
discretization:

\begin{subequations}
  \label{eq:discrete_problem_sc}
  \begin{align}
    \label{eq:discrete_problem_sc_a}
    \mu a_h(\boldsymbol{u}_h, \boldsymbol{v}_h)
    + o_h(u_h; \boldsymbol{u}_h, \boldsymbol{v}_h)
    + \gamma g_h(\boldsymbol{u}_h, \boldsymbol{v}_h)
    + b_h(\boldsymbol{p}_h, v_h)
    &=
      (f, v)_{\mathcal{T}},
    \\
    \label{eq:discrete_problem_sc_b}
    b_h(\boldsymbol{q}_h, u_h)
    &= 0.
  \end{align}
\end{subequations}
Since $F_{uu}$, $B_{pu}$, and $G_{uu}$ are block diagonal with one
block per cell, we can directly eliminate $u$ and $p$ from
\cref{eq:discrete_problem_sc} for a reduced system defined only on
faces. The corresponding condensed preconditioner is given by
\begin{equation}
  \label{eq:condensed2form_up_pre-J}
  \overline{\mathbb{P}}_G
  =
  \begin{bmatrix}
    \bar{F}^g &
    0 \\
    -B_{\bar{p}u}\mathcal{P}^g(F_{uu}+\gamma G_{uu})^{-1}F_{u\bar{u}}^g &
    -\gamma^{-1}\bar{M} - B_{\bar{p}u}\mathcal{P}^g(F_{uu}+\gamma G_{uu})^{-1}B_{\bar{p}u}^T
  \end{bmatrix},
\end{equation}
where
\begin{equation*}
  \begin{split}
    \bar{F}^g
    &= -F_{\bar{u}u}^g\mathcal{P}^g(F_{uu}+\gamma G_{uu})^{-1}F_{u\bar{u}}^g + F_{\bar{u}\bar{u}}^g,
    \\
    \mathcal{P}^g
    &= I - (F_{uu} + \gamma G_{uu})^{-1}B_{pu}^T(B_{pu}(F_{uu}+\gamma G_{uu})^{-1}B_{pu}^T)^{-1}B_{pu},
  \end{split}
\end{equation*}
and $F_{\bar{u}u}^g = F_{\bar{u}u} + \gamma G_{\bar{u}u}$,
$F_{u\bar{u}}^g = F_{u\bar{u}} + \gamma G_{\bar{u}u}^T$,
$F_{\bar{u}\bar{u}}^g = F_{\bar{u}\bar{u}} + \gamma
G_{\bar{u}\bar{u}}$.

As discussed in \Cref{ss:preconditioningHDG:duu}, the Schur complement
in $\bar{p}$ is identical before and after static condensation. Thus
one could also first statically condense the augmented system, and
then note that the Schur complement in $\bar{p}$ is well approximated
by a mass-matrix, resulting in the similar preconditioner
\begin{equation}
  \label{eq:condensed2form_up_pre-J2}
  \overline{\mathbb{P}}_{G,M}
  =
  \begin{bmatrix}
    \bar{F}^g &
    0 \\
    -B_{\bar{p}u}\mathcal{P}^g(F_{uu}+\gamma G_{uu})^{-1}F_{u\bar{u}}^g &
    -\gamma^{-1}\bar{M}
  \end{bmatrix}.
\end{equation}

\begin{remark}
  \label{rem:dhvsgh}
  While the solution
  $(\boldsymbol{u}_h, \boldsymbol{p}_h) \in \boldsymbol{V}_h \times
  \boldsymbol{Q}_h$ to \cref{eq:discrete_problem} and
  \cref{eq:discrete_problem_ALg} are the same, the solution to
  \cref{eq:discrete_problem} and \cref{eq:discrete_problem_sc} are not
  due to \cref{eq:NGvsND}. However, $g_h$ is a consistent penalization
  term that does not affect the well-posedness and accuracy of the
  discretization; see Appendix \ref{ap:wellposedapriori}.
\end{remark}

\subsection{Solving the augmented velocity block}
\label{sec:spack}

The preconditioners $\overline{\mathbb{P}}_G$ in
\cref{eq:condensed2form_up_pre-J} and $\overline{\mathbb{P}}_{G,M}$ in
\cref{eq:condensed2form_up_pre-J2} still require the (action of the)
inverse of $\bar{F}^g$ and either
$-\gamma^{-1}\bar{M} - B_{\bar{p}u}\mathcal{P}^g(F_{uu}+\gamma
G_{uu})^{-1}B_{\bar{p}u}^T$ or $-\gamma^{-1}\bar{M}$. The two trace
pressure Schur complement terms are relatively well-conditioned and
easy to solve. Conversely, by creating simple (and effective)
approximate Schur complements, we have transferred much of this
difficulty to solving the augmented velocity block.

Before static condensation, the augmented velocity block corresponds
to an advection dominated advection-diffusion equation with a large
symmetric singular perturbation. Each of these classes of problems,
advection-dominated problems and singular perturbations, are
independently quite challenging to solve in a fast and scalable
way. When applicable, multilevel methods are typically the most
efficient linear solvers for large sparse systems, particularly
arising from differential operators. Geometric multigrid (GMG) methods
applied to advection-dominated problems almost always require
semi-coarsening and/or line/plane relaxation to capture the advection
on coarse grids. Such techniques require structured grids and provide
poor parallel scaling, as the line/plane relaxation typically spans
many processes. The only exception we are aware of is the
patch-relaxation multigrid methods in \cite{Farrell:2019,Farrell:2021}
that prove robust for certain discretizations of high-Reynolds number
incompressible flow problems, using a local patch-based relaxation
rather than full line/plane relaxation. Why these methods are
effective on high-Reynolds number incompressible problems remains an
open question. However, it should be pointed out that to the best of
our knowledge the preconditioner and patch-based GMG in
\cite{Farrell:2019,Farrell:2021} are the only fast solvers in the
literature that have demonstrated robust performance in mesh spacing
and Reynolds number, and are naturally parallelizable.

The extension of GMG methods to HDG discretizations is nontrivial
\cite{Cockburn:2014a,Fu:2024,Lu:2022,Muralikrishnan:2020}. We are also
interested in algebraic approaches for practical reasons, where
large-scale codes, particularly of the multiphysics type, do not
always maintain a full mesh hierarchy as needed for GMG.  Recent work
has developed fast AIR-algebraic MG (AMG) methods for certain
advection-dominated problems \cite{Manteuffel:2018,Manteuffel:2019},
including HDG advection-diffusion systems \cite{Sivas:2021}. However,
these methods rely heavily on the matrix structure arising from an
advection-dominated problem, and a large symmetric singular
perturbation ruins this structure and the corresponding convergence of
AIR-AMG. A recent paper on preconditioning singular perturbations
$A + \gamma UU^T$ \cite{Benzi:2024} offered one route to combine
AIR-AMG with solvers directly for the singular perturbation, e.g.\
\cite{Dobrev:2019,Lee:2017}.  However, we have not found the method
proposed in \cite{Benzi:2024} to be robust for our
problems. Specifically, the performance of the preconditioner proposed
in \cite{Benzi:2024} appears to be highly sensitive to the linear
system right-hand side as illustrated in \cref{tab:benzi}.

\begin{table}
  \caption{Representative iteration counts for application of the
    preconditioner proposed in \cite{Benzi:2024} to the velocity
    block of a Q2-Q1 Taylor--Hood discretization of a lid-driven
    cavity. Here we use a block lower triangular preconditioner with
    ``exact'' Schur complement up to numerical precision (see
    Appendix \ref{app:PETSc} for details on PETSc KSP Schur complement
    preconditioning). Each KSP preconditioner application requires
    two velocity block solves, one that is directly part of the lower triangular
    KSP preconditioner ($A$ in \eqref{eq:triS}), and one inner solve
    required to compute the action of the Schur complement (see
    $S$ in \eqref{eq:S}) when computing a residual in the KSP.
    For each simulation case, two numbers are reported unless
    GMRES did not converge (DNC). The first is the number of iterations required
    for the velocity block solve in the lower triangular KSP
    preconditioner. The second is the number of iterations required
    for the inner KSP velocity block solves to compute the action of
    the Schur complement. The linear solver relative tolerance is $10^{-2}$
    for each KSP with restart after 300 Krylov basis vectors. When
    applied to the inner KSP velocity solves, the
    preconditioner performs exceptionally well, but when applied to
    the lower KSP velocity solves, the preconditioner
    performs poorly except for very small problems. Moreover, the velocity
    block solves in the lower triangular KSP are arguably the most
    important, as these are fundamental to any block preconditioner,
    even when using an inexact Schur complement as in \eqref{eq:triS}
    and used in numerical results in \cref{sec:num_examples}. }
  {\fontsize{9}{10.5}\selectfont
    \begin{center}
      \begin{tabular}{cc|cccccc}
        \hline
        Cells & DOFs & $\text{Re}=10$ & $\text{Re}=10^2$ & $\text{Re}=5\cdot 10^2$ \\
        \hline
        \num{16}   & \num{187}   & 16,2   & 18,2 & 24,2 \\
        \num{64}   & \num{659}   & 58,2   & DNC,2 & DNC,2 \\
        \num{256}  & \num{2467}  & 177,2  & DNC,2 & DNC,2 \\
        \hline
      \end{tabular}
    \label{tab:benzi}
  \end{center}
}
\end{table}

Post static condensation, the equation is arguably more complicated
because the two distinct components are intertwined in complex
ways. This makes it challenging to take advantage of certain known
structure such as the symmetric singular perturbation with known
kernel.  Other algebraic preconditioners often used in a black-box
manner are approximate sparse inverse methods, such as incomplete LU
(ILU). Unfortunately, our experience has been that ILU is typically
not effective on advection-dominated problems. For example, on the 2D
lid-driven cavity problem we consider in \Cref{sec:num_examples},
ILU(10) converges well for a Reynolds number of 1, but at a Reynolds
number of 10, GMRES with no restart preconditioned by ILU(10) does not
converge in 300 iterations, and performance further degrades with
increasing Reynolds number.

Recently, a new class of methods based on multifrontal sparse LU
solvers using butterfly compression have been developed in the
STRUMPACK library, targeting wave equations and sparse systems where
the frontal matrices are effectively high rank and not amenable to the
low rank approximations used for elliptic problems
\cite{Claus:2023,Liu:2021}. These methods are in principle direct
solvers, with inner approximations and memory compression techniques
to accelerate performance and reduce memory requirements. Between
needing an algebraic solver that is effectively black box and solving
low regularity problems with high Reynolds numbers, the methods
developed in \cite{Claus:2023,Liu:2021} are well-suited to the
challenges arising in augmented high Reynolds number flow. Thus, to
solve the statically condensed velocity field from the augmented
system, $\bar{F}^g$, we use GMRES with multifrontal inexact LU by
STRUMPACK. As demonstrated in \Cref{sec:num_examples}, the performance
is quite good and able to solve challenging high Reynolds number
problems efficiently.

\section{Numerical examples}
\label{sec:num_examples}

We use Newton's method with continuation to solve
\cref{eq:discrete_problem_sc}. At each Newton step the linear problems
are statically condensed and the global problems are solved using
FGMRES preconditioned with block preconditioners
$\overline{\mathbb{P}}_G$ in \cref{eq:condensed2form_up_pre-J} or
$\overline{\mathbb{P}}_{G,M}$ in \cref{eq:condensed2form_up_pre-J2},
based on the block lower triangular preconditioner in \eqref{eq:triS}
with approximate Schur complement.'  The condensed velocity block in
the preconditioner is solved using GMRES preconditioned with
STRUMPACK. The Schur complement is solved using FGMRES with diagonal
Jacobi preconditioning. See \cref{fig:solverdiagram} for the solver
diagram. The numerical examples are implemented in MOOSE
\cite{harbour20254} with solver support from PETSc
\cite{petsc-user-ref,petsc-efficient} and STRUMPACK
\cite{Liu:2021,Claus:2023}.

All simulations use the penalty parameter $\alpha = 10 k^2$, with $k$
the degree of the polynomial approximation, and preconditioning
parameter $\gamma = 10^4$. For the numerical examples we will refer to
the Reynolds number which is defined as $\text{Re} := UL/\mu$ with $U$
the characteristic velocity and $L$ the characteristic length scale of
the flow.

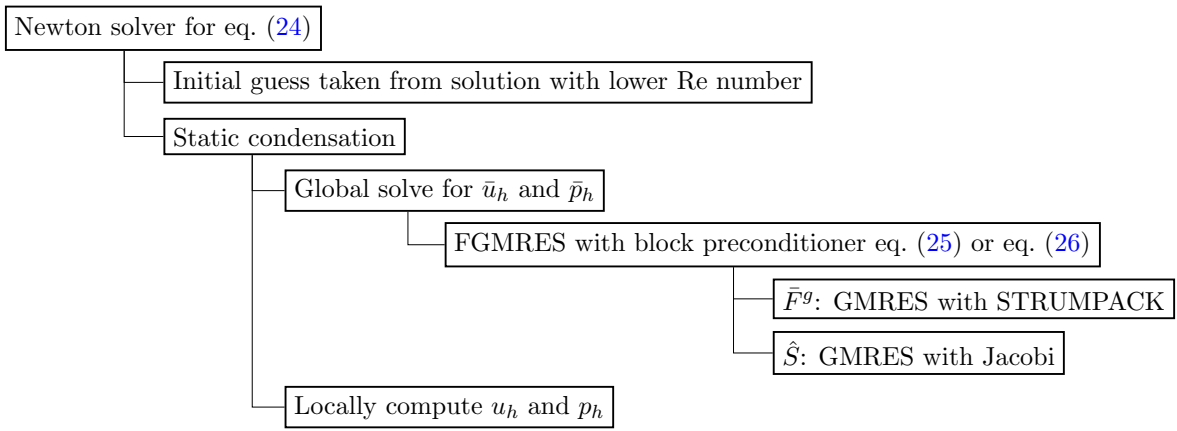
\begin{figure}[tbp]
  \centering
  \scalebox{0.9}{
    \begin{tikzpicture}[%
      every node/.style={draw=black, thick, anchor=west},
      grow via three points={one child at (0.0,-0.8) and
        two children at (0.0,-0.8) and (0.0,-1.6)},
      edge from parent path={(\tikzparentnode.210) |- (\tikzchildnode.west)}]
      \node {Newton solver for \cref{eq:discrete_problem_sc}}
      child {node {Initial guess taken from solution with lower Re number}}
      child {node {Static condensation}
        child {node {Global solve for $\bar{u}_h$ and $\bar{p}_h$}
          child {node {FGMRES with block preconditioner \cref{eq:condensed2form_up_pre-J} or \cref{eq:condensed2form_up_pre-J2}}
            child {node {$\bar{F}^g$: GMRES with STRUMPACK}}
            child {node {$\hat{S}$: GMRES with Jacobi}}
          }
        }
        child[missing]{}
        child[missing]{}
        child[missing]{}
        child {node {Locally compute $u_h$ and $p_h$}}
      };
    \end{tikzpicture}
  }
  \caption{Solver diagram.}
  \label{fig:solverdiagram}
\end{figure}

\subsection{Test cases}
\label{ss:tests}

We consider the lid-driven cavity and backward-facing step
problems. For the lid-driven cavity problem we set
$\Omega = (0, 1)^2$, impose $u = (1, 0)$ on the boundary $x_2=1$ and
set $u=0$ on the remaining boundaries. For the backward-facing step we
consider
$\Omega = \del[1]{ \sbr[0]{0,10} \times \sbr[0]{0,2} } \backslash
\del[1]{\sbr[0]{0,1} \times \sbr[0]{0,1}}$ and impose
$u=(4(2-x_2)(x_2-1),0)$ on the boundary $x_1=0$, outflow boundary at
$x_1=10$, and $u=0$ elsewhere. For both test cases, the characteristic
velocity and characteristic length are unity so that
$\text{Re} = 1/\mu$. The meshes for the lid-driven cavity problem are
regular triangular meshes while the meshes for the backward-facing
step problem have been taken from \cite{FiredrakeZenodo}.

For both test cases we follow \cite{Farrell:2021} and consider
continuation in the Reynolds number, using the previous lower Reynolds
number as initial guess for Newton's method for the current Reynolds
number. For the lid-driven cavity problem we solve for $\text{Re}=1$,
$\text{Re}=10$, $\text{Re}=100$ and then in steps of $250$ until
$\text{Re}=\num{10000}$. For the backward-facing step we solve for
$\text{Re}=1$, $\text{Re}=10$, $\text{Re}=50$, $\text{Re}=100$,
$\text{Re}=150$, $\text{Re}=200$, $\text{Re}=250$, $\text{Re}=300$,
$\text{Re}=350$, $\text{Re}=400$ and then in steps of $200$ until
$\text{Re}=\num{10000}$. The initial guess for $\text{Re}=1$ is the
zero vector. The absolute and relative tolerances for the nonlinear
solver are $10^{-7}$ and $10^{-8}$ respectively. The outer linear
solver relative tolerance is $10^{-4}$ while the inner linear solver
relative tolerances for the condensed velocity block and condensed
Schur complement are both $10^{-2}$. The outer and inner linear solver
absolute tolerances are $10^{-9}$ and $10^{-8}$
respectively. Convergence for both outer and inner linear solvers is
determined based on the unpreconditioned residual norm. Plots of the
solutions at $\text{Re} = \num{10000}$ for the lid-driven cavity and
backward-facing step are given in \cref{fig:graphical-results}.

\begin{figure}[htbp]
  \centering
  \subfloat[Lid-driven cavity.]{\includegraphics[width=0.45\textwidth]{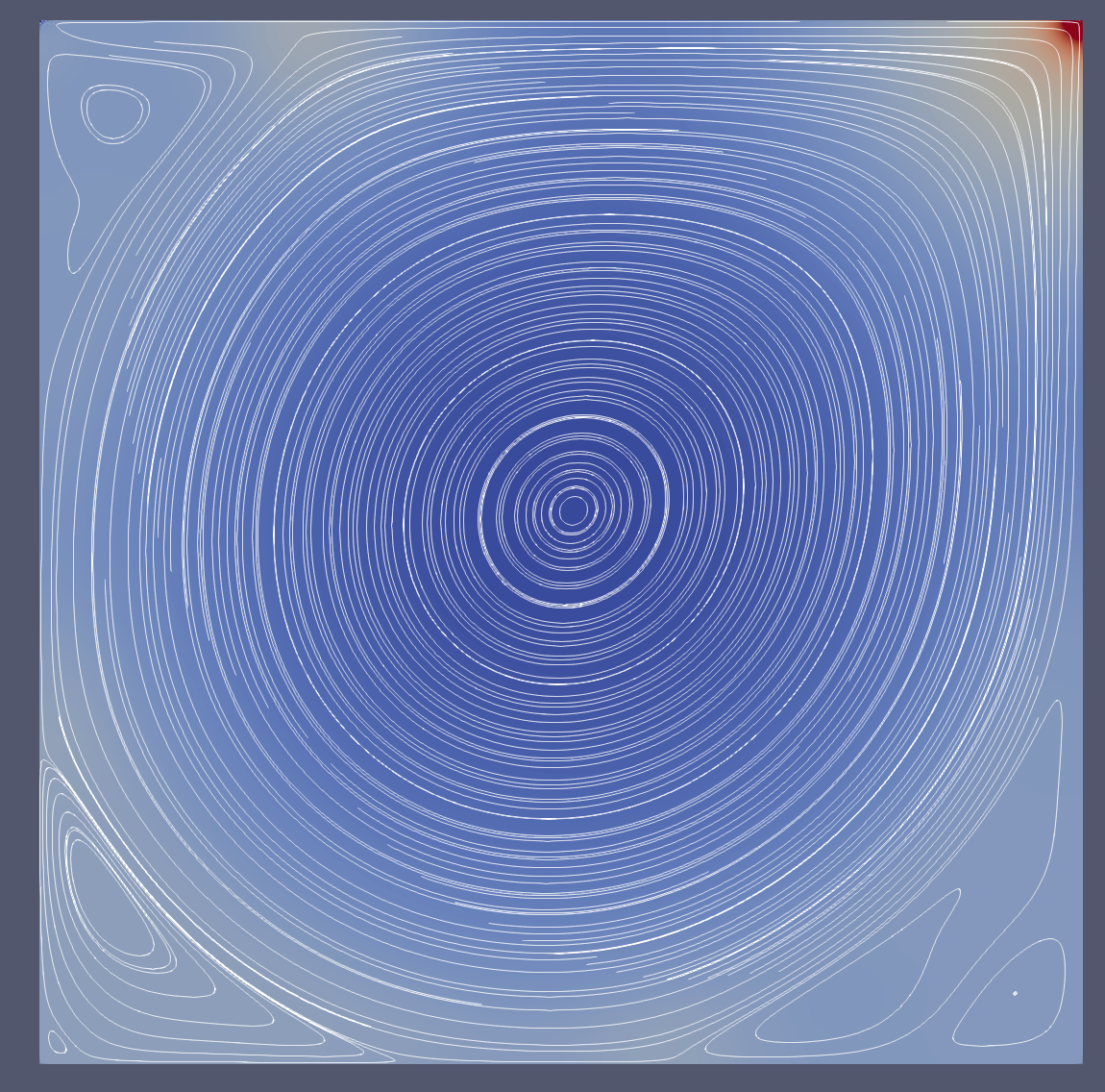}}
  \\
  \subfloat[Backward-facing step.]{\includegraphics[width=0.9\textwidth]{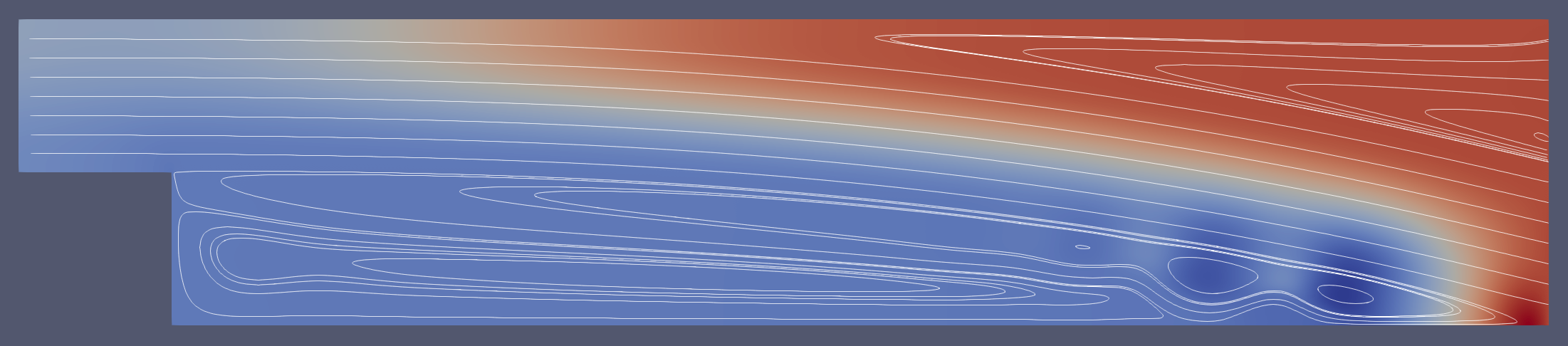}}
  \caption{Pressure colormaps plus streamline plots of the lid-driven cavity and
    backward-facing step problems with $\text{Re} = \num{10000}$}
  \label{fig:graphical-results}
\end{figure}

\subsection{Results}
\label{ss:results}

In \cref{tab:2d_a11_and_mass_hdg} and \cref{tab:2d_mass_hdg} we
present the number of iterations needed to converge for various
Reynolds numbers when $k=2$ and when using the HDG discretization with
\cref{eq:condensed2form_up_pre-J} and
\cref{eq:condensed2form_up_pre-J2} respectively. We report three
numbers which describe the solve (see \cref{fig:solverdiagram}). The
first is the number of Newton iterations. The second is the maximum
number of outer FGMRES iterations observed over the Newton solve which
is an implicit measure of the quality of the Schur complement
approximation. The final number is the maximum number of GMRES
iterations to solve $\bar{F}^g$. We also present the average Newton
solve time, averaged over all Reynolds numbers, divided by the number
of degrees of freedom per process. We observe that the Schur
complement approximation, as measured by the second iteration count,
is robust with respect to both number of degrees of freedom and
Reynolds number. This FGMRES iteration count is significantly lower
for the lid driven case and slightly lower for the backwards facing
step case when using \cref{eq:condensed2form_up_pre-J2} compared to
\cref{eq:condensed2form_up_pre-J}.  We observe that the multifrontal
inexact LU used to solve $\bar{F}^g$, as measured by iteration counts,
is robust with respect to both number of degrees of freedom and the
Reynolds number up to a critical case-dependent problem size. We do
not reach that size for the backward facing step cases considered for
the paper, but we do for the lid driven problem.  For the lid driven
case with $\num{32768}$ cells we observe the number of iterations
required to approximately solve $\bar{F}^g$ increase substantially
(5x) when moving from diffusion to advection dominated regimes. This
notably affects the weak scaling of the solver with the average solve
time increasing by 60\% between $\num{8192}$ cell and $\num{32768}$
cell cases whereas the average solve time only increased by 5\% when
using \cref{eq:condensed2form_up_pre-J} and 30\% when using
\cref{eq:condensed2form_up_pre-J2} at the previous refinement
level. When performing the first refinement for the backwards facing
step we observe the desired weak scaling: constant average solve time
(actually a 10\% decrease). However, at the next refinement level we
note a $\sim 40\%$ increase in average solve time even with a constant
number of iterations required to approximately solve $\bar{F}^g$.  We
attribute this to the non-linear complexity of LU factorization.

\begin{table}
  \caption{Solver results for two dimensional test cases with HDG discretization using the preconditioner \cref{eq:condensed2form_up_pre-J}. The first column is the number of processes.}
  {\fontsize{9.0}{10.5}\selectfont
    \begin{center}
      \begin{tabular}{cccc|cccccccc}
        \hline
        p & Cells & Total dofs & Condensed dofs & $\text{Re}=1$ & $10$ & $10^2$ & $10^3$ & $10^4$ & time/100k local dofs\\
        \hline
        \multicolumn{10}{c}{Lid-driven cavity problem} \\
        \hline
          1 & \num{2048}  & \num{58944} & \num{28224} & 2,10,1 & 2,14,1 & 3,10,1 & 3,10,1 & 2,10,1 & 17.3 \\
          4 & \num{8192}  & \num{234624} & \num{111744} & 2,10,1 & 2,14,1 & 3,10,1 & 3,11,1 & 2,10,1 & 18.2 \\
          16 & \num{32768} & \num{936192} & \num{444672} & 2,10,2 & 2,13,1 & 3,10,11 & 3,10,11 & 1,10,7 & 28.9 \\
        \hline
        \multicolumn{10}{c}{Backward-facing step problem} \\
        \hline
          1 & \num{3203} & \num{92199} & \num{44154} & 2,10,1 & 2,11,1 & 3,11,1 & 3,11,1 & 2,11,1 & 21.5 \\
          4 & \num{12812} & \num{366969} & \num{174789} & 2,10,1 & 2,11,1 & 3,11,1 & 3,11,1 & 2,11,1 & 19.5 \\
          16 & \num{51284} & \num{1464222} & \num{695502} & 2,10,1 & 2,11,1 & 3,12,1 & 3,12,1 & 2,11,1 & 28.0 \\
        \hline
      \end{tabular}
    \label{tab:2d_a11_and_mass_hdg}
  \end{center}
}
\end{table}

\begin{table}
  \caption{Solver results for two dimensional test cases with HDG discretization using the preconditioner \cref{eq:condensed2form_up_pre-J2}. The first column is the number of processes.}
  {\fontsize{9}{10.5}\selectfont
    \begin{center}
      \begin{tabular}{cccc|cccccccc}
        \hline
        p & Cells & Total dofs & Condensed dofs & $\text{Re}=1$ & $10$ & $10^2$ & $10^3$ & $10^4$ & time/100k local dofs\\
        \hline
        \multicolumn{10}{c}{Lid-driven cavity problem} \\
        \hline
          1 & \num{2048}  & \num{58944} & \num{28224} & 2,6,1 & 2,6,1 & 3,6,1 & 3,6,1 & 2,6,1 & 13.2 \\
          4 & \num{8192}  & \num{234624} & \num{111744} & 2,6,1 & 2,6,1 & 3,5,1 & 3,6,1 & 2,5,1 & 17.4 \\
          16 & \num{32768} & \num{936192} & \num{444672} & 2,6,2 & 2,7,3 & 3,7,11 & 3,8,11 & 1,7,9 & 27.3 \\
        \hline
        \multicolumn{10}{c}{Backward-facing step problem} \\
        \hline
          1 & \num{3203} & \num{92199} & \num{44154} & 2,9,1 & 2,10,1 & 3,10,1 & 3,10,1 & 2,10,1 & 21.3 \\
          4 & \num{12812} & \num{366969} & \num{174789} & 2,9,1 & 2,10,1 & 3,10,1 & 3,10,1 & 2,10,1 & 19.2 \\
          16 & \num{51248} & \num{1464222} & \num{695502} & 2,9,1 & 2,11,1 & 3,11,1 & 3,10,1 & 2,10,1 & 27.4 \\
        \hline
      \end{tabular}
    \label{tab:2d_mass_hdg}
  \end{center}
}
\end{table}

The preconditioners in this manuscript are introduced for the HDG
discretization of the Navier--Stokes equations. However, they can also
be applied to an embedded-hybridizable discontinuous Galerkin
(EDG-HDG) \cite{Rhebergen:2020} discretization in which $\bar{V}_h$ in
\cref{eq:femVbh} is replaced by $\bar{V}_h \cap C^0(\Gamma^0)$. In
\cref{tab:2d_a11_and_mass_edghdg} we present results using the EDG-HDG
discretization combined with the preconditioner
\cref{eq:condensed2form_up_pre-J2}. The results are qualitatively
similar to those for HDG. For the lid driven cavity case we note that
the number of strumpack iterations required to solve $\bar{F}^g$ for
the most refined case is lower due to the reduced number of degrees of
freedom present in the EDG-HDG system and consequently lower
compression loss. We note a 100\% increase in the average solve time
when moving to the most refined backwards facing step case. We
attribute this to the increase in number of iterations required to
solve $\bar{F}^g$ as well as the non-linear complexity of LU
factorization.

\begin{table}
  \caption{Solver results for two dimensional test cases with EDG-HDG discretization using the preconditioner \cref{eq:condensed2form_up_pre-J2}. The first column is the number of processes.}
  {\fontsize{9}{10.5}\selectfont
    \begin{center}
      \begin{tabular}{cccc|cccccccc}
        \hline
        p & Cells & Total dofs & Condensed dofs & $\text{Re}=1$ & $10$ & $10^2$ & $10^3$ & $10^4$ & time/100k local dofs\\
        \hline
        \multicolumn{10}{c}{Lid-driven cavity problem} \\
        \hline
          1 & \num{2048} & \num{48578} & \num{17858} & 2,6,1 & 2,6,1 & 3,6,1 & 3,6,1 & 2,6,1 & 16.7 \\
          4 & \num{8192} & \num{193410} & \num{70530} & 2,6,1 & 2,5,1 & 3,4,1 & 3,6,1 & 2,5,1 & 14.3 \\
          16 & \num{32768} & \num{771842} & \num{280322} & 2,6,1 & 2,5,1 & 3,5,2 & 3,7,2 & 1,7,3 & 18.2 \\
        \hline
        \multicolumn{10}{c}{Backward-facing step problem} \\
        \hline
          1 & \num{12812} & \num{302505} & \num{110325} & 2,9,1 & 2,10,1 & 3,10,1 & 3,10,1 & 2,10,1 & 12.6 \\
          4 & \num{51284} & \num{1207172} & \num{438452} & 2,9,1 & 2,10,1 & 3,11,1 & 3,10,1 & 2,10,1 & 15.6 \\
          16 & \num{204992} & \num{4822998} & \num{1748118} & 2,8,2 & 2,10,3 & 3,11,5 & 2,11,10 & 2,10,18 & 30.5 \\
        \hline
      \end{tabular}
    \label{tab:2d_a11_and_mass_edghdg}
  \end{center}
}
\end{table}

\subsection{Further discussion}
\label{ss:discussion}

As stated above, the preconditioners developed in this work are robust
until reaching a critical threshold of problem size and Reynolds
number.  We believe this is an important observation; running the
backwards facing step with a Reynolds number of $6\times 10^5$ using a
$k-\epsilon$ turbulence model yields an effective Reynolds number
(when including the effect of the turbulent viscosity) of
$9\times 10^2$
\cite{mohammadi1993analysis,lindsay2023moose}. Moreover, wall
functions in Reynolds-Averaged Navier--Stokes (RANS) models like
$k-\epsilon$ generally constrain the minimum element size near
boundaries such that the overall problem size is not large
\cite{launder1983numerical,versteeg2007introduction,patankar2018numerical};
we observe strong performance for a large range of Reynolds numbers
for moderate problem sizes. Large Eddy Simulation (LES) models
\cite{smagorinsky1963general,deardorff1970numerical,sagaut2006large},
another possible target area for HDG, have much higher effective
Reynolds numbers than RANS, but in this case simulations are transient
with sufficiently small timesteps such that explicit time integration
can be used or such that the implicit system is significantly easier
to solve due to increasing diagonal dominance. In summary we expect
the algebraic preconditioner proposed here to function effectively for
realistic CFD simulation conditions.

\section{Conclusions}
\label{sec:conclusions}

We introduce an AL-like block preconditioner for a hybridizable
discontinuous Galerkin discretization of the linearized Navier--Stokes
equations at high Reynolds number. New linear algebra theory related
to AL preconditioning is used to motivate the new preconditioners, and
sparse direct solvers from STRUMPACK are used to solve the inner
blocks of the block coupled system. The nonlinearities are resolved
using Newton iterations with continuation in Reynolds number, and
results are demonstrated on the lid-driven cavity and backward facing
step problems. The block preconditioner is shown to be highly robust
and effective, and STRUMPACK solvers are demonstrated to be moderately
robust, most importantly on realistic Reynolds numbers that would
arise in CFD simulations with implicit time stepping or where steady
state solutions make sense.

\section*{Acknowledgments}

BSS was supported by the DOE Office of Advanced Scientific Computing
Research Applied Mathematics program through Contract
No. 89233218CNA000001. Los Alamos National Laboratory Report
LA-UR-25-31645. SR was supported by the Natural Sciences and
Engineering Research Council of Canada through the Discovery Grant
program (RGPIN-2023-03237). ADL was supported by Lab Directed Research
and Development program funding at Idaho National Laboratory under
contract DE-AC07-05ID14517.

The authors thank Abdullah Ali Sivas for discussions on solvers for
HDG methods. The authors also thank Pierre Jolivet for extensive
discussion and experimentation on preconditioning for the augmented
and condensed velocity block.

\bibliographystyle{plain}
\bibliography{references}
\appendix
\section{Proof of \Cref{lem:factDuu}}
\label{app:proofs}

Consider the following two problems: Find $u_h^{\gamma} \in V_h$ such
that
\begin{equation}
  \label{eq:problem1D}
  (u_h^{\gamma}, v_h)_{\mathcal{T}}
  + \gamma \langle h_F^{-1}\jump{ u_h^{\gamma} \cdot n}, \jump{v_h \cdot n} \rangle_{\mathcal{F}}
  = (g, v_h)_{\mathcal{T}}
  \quad \forall v_h \in V_h,
\end{equation}
and: Find $(u_h, \bar{p}_h) \in V_h \times \bar{Q}_h$ such that
\begin{subequations}
  \label{eq:problem2D}
  \begin{align}
    \label{eq:problem2D-a}
    (u_h, v_h)_{\mathcal{T}} + \langle \bar{p}_h, \jump{v_h \cdot n} \rangle_{\mathcal{F}}
    &= (g, v_h)_{\mathcal{T}} && \forall v_h \in V_h,
    \\
    \label{eq:problem2D-b}
    \langle \jump{u_h \cdot n}, \bar{q}_h \rangle_{\mathcal{F}}
    - \gamma^{-1} \langle h_F \bar{p}_h, \bar{q}_h \rangle_{\mathcal{F}}
    &= 0 && \forall \bar{q}_h \in \bar{Q}_h.
  \end{align}
\end{subequations}
\textbf{Step 1.} We will start by proving that \cref{eq:problem1D} and
\cref{eq:problem2D} are equivalent in the sense that if
$(u_h,\bar{p}_h) \in V_h \times \bar{Q}_h$ solves \cref{eq:problem2D}
and $u_h^{\gamma} \in V_h$ solves \cref{eq:problem1D}, then
$u_h = u_h^{\gamma}$ and
$\bar{p}_h = \gamma h_F^{-1} \jump{u_h^{\gamma} \cdot n}$. Subtract
\cref{eq:problem2D-a} from \cref{eq:problem1D}:
\begin{equation*}
  (u_h^{\gamma} - u_h, v_h)_{\mathcal{T}}
  + \gamma \langle h_F^{-1} \jump{u_h^{\gamma} \cdot n}, \jump{v_h \cdot n} \rangle_{\mathcal{F}}
  = \langle \bar{p}_h, \jump{v_h \cdot n} \rangle_{\mathcal{F}}
  \quad \forall v_h \in V_h.
\end{equation*}
Choose $v_h = u_h^{\gamma} - u_h$, then
\begin{equation}
  \label{eq:problem1bmin2avhuhexpD}
  \norm[0]{u_h^{\gamma} - u_h}_{\Omega}^2
  + \gamma \langle h_F^{-1} \jump{u_h^{\gamma} \cdot n}, \jump{u_h^{\gamma} \cdot n} \rangle_{\mathcal{F}}
  - \gamma \langle h_F^{-1} \jump{u_h^{\gamma} \cdot n}, \jump{u_h \cdot n} \rangle_{\mathcal{F}}
  = \langle \bar{p}_h, \jump{u_h^{\gamma} \cdot n} \rangle_{\mathcal{F}}
  - \langle \bar{p}_h, \jump{u_h \cdot n} \rangle_{\mathcal{F}}.
\end{equation}
By choosing $\bar{q}_h = h_F^{-1}\jump{u_h^{\gamma} \cdot n}$ and
$\bar{q}_h = h_F^{-1}\jump{u_h \cdot n}$ in \cref{eq:problem2D-b} we
find, respectively,
\begin{equation*}
  \langle\bar{p}_h, \jump{u_h^{\gamma} \cdot n} \rangle_{\mathcal{F}}
  =
  \gamma \langle \jump{u_h \cdot n}, h_F^{-1}\jump{u_h^{\gamma} \cdot n} \rangle_{\mathcal{F}},
  \quad
  \langle\bar{p}_h, \jump{u_h \cdot n} \rangle_{\mathcal{F}}
  =
  \gamma \langle \jump{u_h \cdot n}, h_F^{-1} \jump{u_h \cdot n} \rangle_{\mathcal{F}}.
\end{equation*}
Using these expressions in \cref{eq:problem1bmin2avhuhexpD},
\begin{equation*}
  \norm[0]{u_h^{\gamma} - u_h}_{\Omega}^2
  + \gamma \langle h_F^{-1} \jump{u_h^{\gamma} \cdot n}, \jump{u_h^{\gamma} \cdot n} \rangle_{\mathcal{F}}
  - 2\gamma \langle h_F^{-1} \jump{u_h^{\gamma} \cdot n}, \jump{u_h \cdot n} \rangle_{\mathcal{F}}
  =
  - \gamma \langle h_F^{-1} \jump{u_h \cdot n}, \jump{u_h \cdot n} \rangle_{\mathcal{F}}.
\end{equation*}
We may write this as:
\begin{equation*}
  \norm[0]{u_h^{\gamma} - u_h}_{\Omega}^2
  + \gamma \norm[0]{ h_F^{-1} \jump{(u_h^{\gamma}-u_h) \cdot n} }_{\Gamma^0}^2
  - \gamma \langle h_F^{-1} \jump{u_h \cdot n}, \jump{u_h \cdot n} \rangle_{\mathcal{F}}
  = - \gamma \langle h_F^{-1} \jump{u_h \cdot n}, \jump{u_h \cdot n} \rangle_{\mathcal{F}},
\end{equation*}
in other words,
$\norm[0]{u_h^{\gamma} - u_h}_{\Omega}^2 + \gamma \norm[0]{h_F^{-1}
  \jump{(u_h^{\gamma}-u_h) \cdot n} }_{\Gamma^0}^2 = 0$ implying that
$u_h = u_h^{\gamma}$. To now show that
$\bar{p}_h = \gamma h_F^{-1} \jump{u_h^{\gamma} \cdot n}$, use
$u_h = u_h^{\gamma}$ in \cref{eq:problem2D-b} to find that
$\langle \gamma \jump{u_h^{\gamma} \cdot n} - h_F \bar{p}_h, \bar{q}_h
\rangle_{\mathcal{F}_h} = 0$ for all $\bar{q}_h \in \bar{Q}_h$.
Choosing
$\bar{q}_h = \gamma \jump{u_h^{\gamma} \cdot n} - h_F \bar{p}_h$ we
obtain
$\norm[0]{\gamma \jump{u_h^{\gamma} \cdot n} - h_F
  \bar{p}_h}_{\Gamma^0}^2 = 0$ implying that
$\bar{p}_h = \gamma h_F^{-1} \jump{u_h^{\gamma} \cdot n}$.

\textbf{Step 2.} With the equivalence between \cref{eq:problem1D} and
\cref{eq:problem2D} established, we write both problems in matrix
form. \Cref{eq:problem1D} in matrix form is given by:
\begin{equation}
  \label{eq:uonpartialKdiscD}
  (M_u + \gamma D_{uu})u = G,
\end{equation}
where $M_u$ is the mass matrix on the cell velocity space. Noting that
$\langle \bar{q}_h, \jump{v_h \cdot n} \rangle_{\mathcal{F}} = \langle
\bar{q}_h, v_h \cdot n \rangle_{\partial \mathcal{T}} = b_2(\bar{q}_h,
v_h)$, \cref{eq:problem2D} in matrix form is given by:
\begin{equation}
  \label{eq:uonpartialKmixeddiscD}
  \begin{bmatrix}
    M_u & B_{\bar{p}u}^T \\ B_{\bar{p}u} & - \gamma^{-1}\bar{M}
  \end{bmatrix}
  \begin{bmatrix}
    u \\ \bar{p}
  \end{bmatrix}
  =
  \begin{bmatrix}
    G \\ 0
  \end{bmatrix}.
\end{equation}
Eliminating $\bar{p}$ from \cref{eq:uonpartialKmixeddiscD}, we obtain
the following equation for $u$:
\begin{equation}
  \label{eq:elimbarpmixeddisconpartialKD}
  (M_u + \gamma B_{\bar{p}u}^T \bar{M}^{-1} B_{\bar{p}u}) u = G.
\end{equation}
The result follows by comparing \cref{eq:uonpartialKdiscD} and
\cref{eq:elimbarpmixeddisconpartialKD}.

\section{Well-posedness and a priori error analysis}
\label{ap:wellposedapriori}

Here we present a well-posedness and a priori error analysis of
\cref{eq:discrete_problem_sc}. For the Navier--Stokes equations
\cref{eq:ns} with homogeneous Dirichlet boundary conditions and
constant viscosity $\mu$ we note that we can replace
$-\nabla\cdot(2\mu\varepsilon(u))$ by $-\mu\nabla^2u$. This
simplification is used in this section to analyze
\cref{eq:discrete_problem_sc}.

Well-posedness and an a priori error analysis of
\cref{eq:discrete_problem_sc} with $\gamma = 0$ was proven in
\cite{Kirk:2019}. Although $g_h(\cdot, \cdot)$ \cref{eq:ghu} is a
consistent penalty term, the nullspace of $g_h$ is a subspace of
$V_h \cap H(\text{div};\Omega)$ and so the solution to
\cref{eq:discrete_problem_sc} will depend on $\gamma$. To determine
the effect of $\gamma$, we generalize the results of \cite{Kirk:2019}
to $\gamma \ge 0$. We will require the following (semi-)norms on the
extended velocity space
$\boldsymbol{V}(h) := \boldsymbol{V}_h + (H_0^1(\Omega)^d\cap
H^2(\Omega)^d) \times H_0^{3/2}(\Gamma^0)^d$:
\begin{align*}
  \tnorm{ \boldsymbol{v} }_{v}^2
  &:= \sum_{K\in\mathcal{T}_h} \del[2]{\norm{\nabla v }^2_{K}
    + h_K^{-1} \norm{\bar{v} - v}^2_{\partial K}},
  &
  \tnorm{ \boldsymbol{v} }_{v'}^2
  &:= \tnorm{ \boldsymbol{v} }_{v}^2 + \sum_{K \in \mathcal{T}_h} h_K\norm[0]{\partial_nv}_{\partial K}^2,
  \\
  \tnorm{ \boldsymbol{v} }_{v,\gamma}^2
  &:= \tnorm{ \boldsymbol{v} }_{v}^2 + \gamma \mu^{-1} |\boldsymbol{v} |_{r}^2,
  &
  \tnorm{ \boldsymbol{v} }_{v',\gamma}^2
  &:= \tnorm{ \boldsymbol{v} }_{v'}^2 + \gamma \mu^{-1} |\boldsymbol{v} |_{r}^2.
\end{align*}
where
$\envert[0]{ \boldsymbol{v} }_{r}^2 := \sum_{K \in \mathcal{T}_h}
h_K^{-1}\norm[0]{(v-\bar{v})\cdot n}_{\partial K}^2$. For
$v \in V(h) := V_h + H_0^1(\Omega)^d \cap H^2(\Omega)^d$ we set
$\norm[0]{v}_{1,h} := \tnorm{(v, \av{v})}_v$ and note that
\begin{equation}
  \label{eq:1hvs1ht}
  \tnorm{(v,\av{v})}_{v,\gamma} = \tnorm{(v,\av{v})}_{v}
  \qquad \forall v \in V(h) \cap H(\text{div};\Omega).
\end{equation}
Furthermore, by \cite[Proposition A.2]{Cesmelioglu:2017},
\begin{equation}
  \label{eq:Poincare}
  \norm[0]{v}_{\Omega} \le c \norm[0]{v}_{1,h} \le c \tnorm{\boldsymbol{v}}_{v}
  \qquad \forall \boldsymbol{v} \in \boldsymbol{V}(h).
\end{equation}
On the extended pressure space
$\boldsymbol{Q}(h) := \boldsymbol{Q}_h + (L_0^2(\Omega) \cap
H^1(\Omega)) \times H_0^{1/2}(\Gamma^0)$ we define the norm
\begin{equation*}
  \tnorm{\boldsymbol{q}}_p^2
  := \norm[0]{q}_{\Omega}^2 + \sum_{K \in \mathcal{T}_h} h_K\norm[0]{\bar{q}}_{\partial K}^2.
\end{equation*}
Finally, we define
\begin{equation*}
  \tnorm{(\boldsymbol{v}_h,\boldsymbol{q}_h)}_{v,\gamma,p}^2
  := \mu \tnorm{\boldsymbol{v}_h}_{v,\gamma}^2 + \mu^{-1} \tnorm{\boldsymbol{q}_h}_p^2
  \qquad \forall (\boldsymbol{v}_h,\boldsymbol{q}_h) \in \boldsymbol{V}_h \times \boldsymbol{Q}_h.
\end{equation*}

Note that
\begin{align*}
  a_h(\boldsymbol{v}_h, \boldsymbol{v}_h)
  &\ge c \tnorm{\boldsymbol{v}_h}_{v}^2
  && \forall \boldsymbol{v}_h \in \boldsymbol{V}_h,
  \quad
    a_h(\boldsymbol{u}, \boldsymbol{v})
  \le c \tnorm{\boldsymbol{u}}_{v'} \tnorm{\boldsymbol{v}}_{v'}
  && \forall \boldsymbol{u}, \boldsymbol{v} \in \boldsymbol{V}(h),
  \\
  g_h(\boldsymbol{v}_h, \boldsymbol{v}_h)
  &= |\boldsymbol{v}_h|_r^2 \ge 0
  && \forall \boldsymbol{v}_h \in \boldsymbol{V}_h,
  \quad
  |g_h(\boldsymbol{u}, \boldsymbol{v})|
  \le |\boldsymbol{u}|_r |\boldsymbol{v}|_r
  && \forall \boldsymbol{u}, \boldsymbol{v} \in \boldsymbol{V}(h),
\end{align*}
where the inequalities in the first row are shown in \cite[Lemmas 4.2
and 4.3]{Rhebergen:2017}. Combining these inequalities, and using
H\"older's inequality, we note that
\begin{align*}
  \mu a_h(\boldsymbol{u}, \boldsymbol{v}) + \gamma g_h(\boldsymbol{u}, \boldsymbol{v})
  &\le c \mu \tnorm{\boldsymbol{u}}_{v',\gamma} \tnorm{\boldsymbol{v}}_{v',\gamma}
  && \forall \boldsymbol{u}, \boldsymbol{v} \in \boldsymbol{V}(h),
  \\
  \mu a_h(\boldsymbol{v}_h, \boldsymbol{v}_h) + \gamma g_h(\boldsymbol{v}_h, \boldsymbol{v}_h)
  &\ge c \mu \tnorm{\boldsymbol{v}_h}_{v,\gamma}^2
  && \forall \boldsymbol{v}_h \in \boldsymbol{V}_h.
\end{align*}
By \cite[Proposition 3.6]{Cesmelioglu:2017} we have for
$w_h \in \cbr[0]{v_h \in V_h \,:\, b_h(\boldsymbol{q}_h,v_h) = 0 \
  \forall \boldsymbol{q}_h \in \boldsymbol{Q}_h}$,
\begin{equation*}
  o_h(w_h;\boldsymbol{v}_h, \boldsymbol{v}_h) \ge 0
  \qquad \forall \boldsymbol{v}_h \in \boldsymbol{V}_h,
\end{equation*}
and by \cite[Proposition 3.4]{Cesmelioglu:2017}, for
$w_1, w_2 \in V(h)$, $\boldsymbol{u} \in \boldsymbol{V}(h)$, and
$\boldsymbol{v} \in \boldsymbol{V}(h)$, we have
\begin{equation*}
  |o_h(w_1;\boldsymbol{u},\boldsymbol{v}) - o_h(w_2;\boldsymbol{u},\boldsymbol{v})|
  \le c \norm[0]{w_1-w_2}_{1,h} \tnorm{\boldsymbol{u}}_{v} \tnorm{\boldsymbol{v}}_{v}.
\end{equation*}
For $b_h$ we have that for all $\boldsymbol{v} \in \boldsymbol{V}(h)$
and $\boldsymbol{q} \in \boldsymbol{Q}(h)$ \cite[Lemma
4.8]{Rhebergen:2017}:
\begin{equation*}
  |b_h(\boldsymbol{q},v)| \le c \tnorm{\boldsymbol{v}}_{v} \tnorm{\boldsymbol{q}}_p
  \le c \tnorm{\boldsymbol{v}}_{v,\gamma} \tnorm{\boldsymbol{q}}_p.
\end{equation*}
An inf-sup condition for $b_h$ in terms of $\tnorm{\cdot}_{v}$ was
proven in \cite[Lemma 1]{Rhebergen:2018b}. The following lemma proves
an inf-sup condition with respect to $\tnorm{\cdot}_{v,\gamma}$.

\begin{lemma}[inf-sup condition]
  There exists a constant $c > 0$, independent of $h, \mu, \gamma$,
  such that
  \begin{equation*}
    \sup_{\boldsymbol{0}\ne \boldsymbol{v}_h \in \boldsymbol{V}_h}
    \frac{b_h(\boldsymbol{q}_h,v_h)}{\tnorm{\boldsymbol{v}_h}_{v,\gamma}}
    \ge
    c \tnorm{\boldsymbol{q}_h}_p \quad \forall \boldsymbol{q}_h \in \boldsymbol{Q}_h.
  \end{equation*}
\end{lemma}
\begin{proof}
  \textbf{Step 1.} Define the following two spaces:
  \begin{align*}
    \textbf{Ker}(g_h)
    &:= \cbr[0]{\boldsymbol{v}_h \in
      \boldsymbol{V}_h \,:\, ((v_h-\bar{v}_h)\cdot n)|_F = 0\ \forall F
      \in \mathcal{F}_h},
    \\
    \textbf{Ker}(b_2)
    &:= \cbr[0]{\boldsymbol{v}_h \in \boldsymbol{V}_h
      \,:\, v_h \in H(\text{div};\Omega)},
  \end{align*}
  Let $\Pi_V : H^1(\Omega)^d \to V_h$ be the BDM interpolation
  operator \cite[Section III.3]{Brezzi:book}. Note that for all
  $v \in H^1(\Omega)^d$ we have that
  $(\Pi_Vv, \av{\Pi_Vv}) \in \textbf{Ker}(g_h)$ and, using
  \cref{eq:1hvs1ht},
  \begin{equation}
    \label{eq:pivnormseq}
    \tnorm{(\Pi_Vv, \av{\Pi_Vv})}_{v,\gamma}
    \le c \norm[0]{\Pi_Vv}_{1,h}
    \le c \norm[0]{v}_{1,\Omega},
  \end{equation}
  where the second inequality was shown in \cite[Proposition
  10]{Hansbo:2002}. For a $q_h \in Q_h$ there exists a
  $v_{q_h} \in H_0^1(\Omega)^d$ such that $\nabla \cdot v_{q_h} = q_h$
  and $c \norm[0]{v_{q_h}}_{1,\Omega} \le \norm[0]{q_h}_{\Omega}$
  (see, for example, \cite[Theorem 6.5]{Pietro:book}). Using
  \cref{eq:pivnormseq},
  $\tnorm{(\Pi_Vv_{q_h}, \av{\Pi_Vv_{q_h}})}_{v,\gamma} \le c
  \norm[0]{q_h}_{\Omega}$. We find:
  \begin{equation}
    \label{eq:infsupqb0}
    \sup_{\boldsymbol{0}\ne \boldsymbol{v}_h \in \textbf{Ker}(b_2)}
    \frac{b_h((q_h,0),v_h)}{\tnorm{\boldsymbol{v}_h}_{v,\gamma}}
    \ge
    \sup_{\boldsymbol{0}\ne \boldsymbol{v}_h \in \textbf{Ker}(g_h)}
    \frac{b_h((q_h,0),v_h)}{\tnorm{\boldsymbol{v}_h}_{v,\gamma}}
    \ge
    \frac{b_h((q_h,0),\Pi_Vv_{q_h})}{\tnorm{(\Pi_Vv_{q_h}, \av{\Pi_Vv_{q_h}})}_{v,\gamma}}
    \ge c \norm[0]{q_h}_{\Omega},
  \end{equation}
  where the first inequality is because
  $\textbf{Ker}(g_h) \subset \textbf{Ker}(b_2)$.
  \\
  \textbf{Step 2.} Let $L^{BDM}:R_k(\partial K) \to P_k(K)^d$ be the
  BDM local lifting operator such that $(L^{BDM}\mu) \cdot n = \mu$
  and
  $\norm[0]{L^{BDM}\mu}_K \le c h_K^{1/2}\norm[0]{\mu}_{\partial K}$
  for all $\mu \in R_k(\partial K)$ (see \cite[Proposition
  2.10]{Du:book}). Then, using \cref{eq:1hvs1ht}, \cref{eq:Poincare},
  and \cite[Eq. (25)]{Rhebergen:2018b},
  \begin{equation*}
    \tnorm{(L^{BDM}\bar{q}_h, \av{L^{BDM}\bar{q}_h})}_{v,\gamma}
    \le c\norm[0]{L^{BDM}\bar{q}_h}_{1,h}
    \le c\tnorm{(L^{BDM}\bar{q}_h,0)}_{v}
    \le c\sum_{K \in \mathcal{T}_h}h_K^{-1/2}\norm[0]{\bar{q}_h}_{\partial K}.
  \end{equation*}
  Identical steps as in \cite[Lemma 3]{Rhebergen:2018b} then results
  in
  \begin{equation}
    \label{eq:infsupb2}
    \sup_{\boldsymbol{0}\ne \boldsymbol{v}_h \in \boldsymbol{V}_h}
    \frac{b_h((0,\bar{q}_h),v_h)}{\tnorm{\boldsymbol{v}_h}_{v,\gamma}}
    \ge
    c \tnorm{(0,\bar{q}_h)}_p.
  \end{equation}
  \textbf{Step 3.} Since
  $\textbf{Ker}(b_2) = \cbr[0]{\boldsymbol{v}_h \in \boldsymbol{V}_h
    \,:\, b_h((0,\bar{q}_h),\boldsymbol{v}_h)=0 \ \forall \bar{q}_h
    \in \boldsymbol{Q}_h}$, the result follows after combining
  \cref{eq:infsupqb0}, \cref{eq:infsupb2}, and \cite[Theorem
  3.1]{Howell:2011}.
\end{proof}

The next lemma proves well-posedness of \cref{eq:discrete_problem_sc}
by using the aforementioned properties of the forms $a_h$, $g_h$,
$o_h$, and $b_h$.

\begin{lemma}[Well-posedness]
  The HDG discretization \cref{eq:discrete_problem_sc} of the
  Navier--Stokes equations is well-posed for $\gamma \ge
  0$. Furthermore, the solution
  $(\boldsymbol{u}_h,\boldsymbol{p}_h) \in \boldsymbol{V}_h \times
  \boldsymbol{Q}_h$ to \cref{eq:discrete_problem_sc} satisfies:
  \begin{equation*}
    \tnorm{\boldsymbol{u}_h}_{v,\gamma} \le c \mu^{-1} \norm[0]{f}_{\Omega},
    \qquad
    \tnorm{(\boldsymbol{u}_h,\boldsymbol{p}_h)}_{v,\gamma,p} \le c \norm[0]{f}_{\Omega} + c\nu^{-2}\norm[0]{f}_{\Omega}^2.
  \end{equation*}
\end{lemma}
\begin{proof}
  The proof is identical to that of \cite[Lemma 1]{Kirk:2019} but with
  $a_h(\cdot, \cdot)$ replaced by
  $a_h(\cdot, \cdot) + \gamma g_h(\cdot, \cdot)$ and $\tnorm{\cdot}_v$
  replaced by $\tnorm{\cdot}_{v,\gamma}$.
\end{proof}

Let $\Pi_V : H^1(\Omega)^d \to V_h$ again be the BDM interpolation
operator. Furthermore, let $\bar{\Pi}_Vu$ be defined such that when
restricted to a face $F$ then $\bar{\Pi}_Vu|_F = \av{\Pi_Vu}$. Define
$e_u^I = u-\Pi_Vu$, $\bar{e}_u^I = \bar{u} - \bar{\Pi}_Vu$, and
$\boldsymbol{e}_u^I = (e_u^I,\bar{e}_u^I)$, where
$\bar{u} = u|_{\Gamma^0}$. Furthermore, let
$\bar{\Pi}_{L^2(\Gamma^0)}$ be the standard $L^2$-projection operator
onto $\bar{V}_h$.

The following lemma now determines error estimates for the discrete
velocity and pressure. We will write
$h := \max_{K \in \mathcal{T}_h} h_K$.

\begin{lemma}[Error estimates]
  Let $(u,p) \in H^{k+1}(\Omega)^d \times H^k(\Omega)$ be the solution
  to the Navier--Stokes equations \cref{eq:ns},
  $\boldsymbol{u}=(u,\bar{u})$, $\boldsymbol{p}=(p,\bar{p})$, and
  $(\boldsymbol{u}_h,\boldsymbol{p}_h) \in \boldsymbol{V}_h \times
  \boldsymbol{Q}_h$ the solution to \cref{eq:discrete_problem_sc} and
  assume that $\norm[0]{f}_{\Omega} \le c \mu^2$. Then
  \begin{equation*}
    \tnorm{\boldsymbol{u} - \boldsymbol{u}_h}_{v,\gamma} \le c h^k \norm[0]{u}_{k+1,\Omega},
    \qquad
    \norm[0]{p-p_h}_{\Omega} \le c h^k \norm[0]{p}_k + c \mu h^k \norm[0]{u}_{k+1,\Omega}.
  \end{equation*}
\end{lemma}
\begin{proof}
  Consider first the velocity error estimate. We have
  \begin{equation}
    \label{eq:triineq}
    \tnorm{\boldsymbol{u} - \boldsymbol{u}_h}_{v,\gamma}
    \le
    c \tnorm{\boldsymbol{e}_u^I}_{v',\gamma}
    =
    c \tnorm{\boldsymbol{e}_u^I}_{v'}
    =
    c \del[2]{\norm[0]{e_u^I}_{1,h}^2 + \sum_{K \in \mathcal{T}_h}\norm[0]{\partial_ne_u^I}_{\partial K}^2}^{1/2}
    \le c \tnorm{(e_u^I,u-\bar{\Pi}_{L^2(\Gamma^0)}u)}_{v'},
  \end{equation}
  where the first inequality can be shown using identical steps as in
  the proof of \cite[Theorem 1]{Kirk:2019}, but with
  $a_h(\cdot, \cdot)$ replaced by
  $a_h(\cdot, \cdot) + \gamma g_h(\cdot, \cdot)$, $\tnorm{\cdot}_v$
  replaced by $\tnorm{\cdot}_{v,\gamma}$, and $\tnorm{\cdot}_{v'}$
  replaced by $\tnorm{\cdot}_{v',\gamma}$. The first equality is by
  \cref{eq:1hvs1ht}, the second equality is by definition of
  $\norm[0]{\cdot}_{1,h}$, and the last inequality is by
  \cref{eq:Poincare}. The result follows by using \cite[Lemma
  9]{Rhebergen:2020}.

  For the pressure error estimate, replace $a_h(\cdot, \cdot)$ by
  $a_h(\cdot, \cdot) + \gamma g_h(\cdot, \cdot)$, $\tnorm{\cdot}_v$ by
  $\tnorm{\cdot}_{v,\gamma}$, and $\tnorm{\cdot}_{v'}$ by
  $\tnorm{\cdot}_{v',\gamma}$ in the proof of \cite[Lemma
  3]{Kirk:2019} to find
  \begin{equation*}
    \norm[0]{p-p_h}_{\Omega} \le c h^k \norm[0]{p}_k + c \mu \tnorm{\boldsymbol{e}_u^I}_{v',\gamma}.
  \end{equation*}
  The result follows after using the same steps as used in
  \cref{eq:triineq} to bound $\tnorm{\boldsymbol{e}_u^I}_{v',\gamma}$.
\end{proof}

\section{More on Schur complement preconditioners and their PETSc implementation}
\label{app:PETSc}

As described in \cite{petsc-user-ref} the inverse of the Schur
complement factorization is implemented in PETSc as
\begin{equation}
  \label{eq:full_schur}
  \begin{pmatrix}
    A_{00}^{-1} & 0\\
    0 & I
  \end{pmatrix}
  \begin{pmatrix}
    I & -A_{01}\\
    0 & I
  \end{pmatrix}
  \begin{pmatrix}
    I & 0\\
    0 & S^{-1}
  \end{pmatrix}
  \begin{pmatrix}
    I & 0\\
    -A_{10}A_{00}^{-1} & I
  \end{pmatrix}
\end{equation}
Inverses are approximated via Krylov subspace solvers--denoted KSP in
PETSc. PETSc offers various levels of completeness of
\cref{eq:full_schur}. One is \texttt{full} which includes all the
components of \cref{eq:full_schur}. The \texttt{lower} factorization
option, which is what is used in all field splits outlined in this
paper, drops a block matrix multiplication, and after some
rearrangement yields:
\begin{equation}
  \label{eq:lower_schur}
  \begin{pmatrix}
    I & 0\\
    0 & S^{-1}
  \end{pmatrix}
  \begin{pmatrix}
    I & 0\\
    -A_{10} & I
  \end{pmatrix}
  \begin{pmatrix}
    A_{00}^{-1} & 0\\
    0 & I
  \end{pmatrix}
\end{equation}
\Cref{eq:lower_schur} explicitly requires only one $A_{00}^{-1}$
approximate inversion via KSP, which is known as the \emph{lower} KSP
after our lower factorization choice. However, we note that
approximate inversion of $S$ via a Krylov subspace solver involves
multiplication by $S$ which itself introduces additional
$\textrm{KSP}(A_{00})$ because
$S = A_{11} - A_{10}\textrm{KSP}(A_{00})A_{01}$. Consequently every
iteration of KSP($S$) involves a nested/inner solve of
KSP($A_{00}$). PETSc allows specifying different preconditioning
options for lower and inner KSP($A_{00}$) solves, but by default it
uses the same preconditioning options for both cases.

\end{document}